\documentclass[10pt,draft,reqno]{amsart}
     \makeatletter
\newtheorem{theorem}{Theorem}[section]
\newtheorem{lemma}[theorem]{Lemma}
\newtheorem{proposition}[theorem]{Proposition}
\newtheorem{corollary}[theorem]{Corollary}
\theoremstyle{definition}
\newtheorem{definition}[theorem]{Definition}

\theoremstyle{remark}
\newtheorem{remark}[theorem]{Remark}
\numberwithin{equation}{section}
\usepackage{xcolor}

\begin{document}

\title[$\mathcal{R}(p,q)-$ analogs of discrete distributions]{$\mathcal{R}(p,q)-$ analogs of discrete distributions: general formalism and application}
\author[Mahouton Norbert Hounkonnou]{Mahouton Norbert Hounkonnou*}
\thanks{* Corresponding author}
\address{Mahouton Norbert Hounkonnou: International Chair in Mathematical Physics
	and Applications(ICMPA-UNESCO Chair), University of Abomey-Calavi, 072 B.P. 50 Cotonou, Republic of Benin}
\email{norbert.hounkonnou@cipma.uac.bj, with copy to hounkonnou@yahoo.fr }
\author{Fridolin Melong}
\address{Fridolin Melong: International Chair in Mathematical Physics
	and Applications(ICMPA-UNESCO Chair), University of Abomey-Calavi, 072 B.P. 50 Cotonou, Republic of Benin}
\email{fridomelong@gmail.com}



\begin{abstract}
In this paper, we define and discuss  $\mathcal{R}(p,q)$- deformations of  basic univariate discrete distributions  of the probability theory. We mainly focus on binomial, Euler, P\'olya and inverse P\'olya   distributions. We discuss relevant $\mathcal{R}(p,q)-$ deformed factorial moments  of a random variable, and establish associated  expressions of mean and variance.  Futhermore, we derive a recursion relation for the probability distributions.  Then, we apply  the same approach to build main distributional  properties characterizing  the  generalized $q-$ Quesne quantum algebra, used in physics. Other known results  in the literature are also recovered as particular cases. 
\end{abstract}

\maketitle
\section{Introduction}
Deformed quantum algebras, namely   $q-$ deformed  algebras \cite{Jimbo, O, Q} and their extensions to  $(p,q)-$analogs \cite{B1,CJ},  continue to attract much attention in mathematics and physics.
 Inspired by the $q-$ deformed quantum algebras,  Chung  and  Kang  developed a notion of $q-$ permutations and $q-$ combinations \cite{CK}.
The $q-$ combinatorics  and $q-$ hypergeometric series were examined in
\cite{Al-Salam,BS, CA3, Kim}.
A theory of $(p,q)-$ analogs of  binomial coefficients and $(p,q)-$ Stirling numbers was elaborated  in \cite{RC, WW}.
In the same vein, from the $q-$ combinatorics, several authors investigated the $q-$ discrete probability distributions such as $q-$ binomial, 
$q-$ Euler, $q-$ hypergeometric, $q-$ P\'olya,  $q-$ contagious, and $q-$ uniform distributions \cite{CA1,AKemp, K2}.
Recently, our research group was interested in the study of  quantum algebras, and, more especially, in the generalization of the well-known $(p,q)-$ Heisenberg algebras.
From  Odzijewicz's  work \cite{O}, we introduced the $\mathcal{R}(p,q)-$ deformed quantum  algebras  as  generalizations of known deformed quantum algebras \cite{HB1}.
In \cite{HB}, we performed the $\mathcal{R}(p,q)-$ differentiation and   integration, and deducted all relevant  particular cases of $q-$ and $(p,q)-$ deformations.
This opens a novel route for developing  the theory of $\mathcal{R}(p,q)-$ analogs of special numbers,  combinatorics,  and probability distributions.
Then,  the following question naturally arises:
How to construct  
discrete  probability distributions and  deduce their  properties  from  $\mathcal{R}(p,q)-$ deformed quantum  algebras?

This paper
is
organized as follows:
In section $2,$ we briefly recall the notions of  $\mathcal{R}(p,q)-$ deformed numbers,  $\mathcal{R}(p,q)-$ deformed factorial, 
$\mathcal{R}(p,q)-$binomial coefficients, $\mathcal{R}(p,q)-$ derivative,  relevant properties from our generalization of $q-$ Quesne algebra,  and some basic notations and definitions. 
Section $3$ is focused on   main results.  $\mathcal{R}(p,q)-$ deformed combinatorics and  $\mathcal{R}(p,q)-$ analogs of discrete distributions are defined. The mean value,  or expectation, and variance of  random variables of  distributions are computed.  
Then, we apply  the same approach to build main distributional  properties characterizing  the  generalized $q-$ Quesne quantum algebra, used in physics. Other known results  are also recovered as particular cases. We end 
with some concluding remarks in Section $4.$ 
\section{Quick overlook on known results}

\subsection{$\mathcal{R}(p,q)-$ deformed numbers} 
			Consider $ p$ and $q,$ two positive real numbers such that $ 0<q<p<1,$ and a given 
			meromorphic function defined on $\mathbb{C}\times\mathbb{C}$ by: \begin{equation}\label{r10}
			\mathcal{R}(u,v)= \sum_{s,t=-l}^{\infty}r_{st}u^sv^t,
			\end{equation}
			where $r_{st}$ are complex numbers, $l\in\mathbb{N}\backslash\left\lbrace 0\right\rbrace,$ $\mathcal{R}(p^n,q^n)>0,  \forall n\in\mathbb{N},$ and $\mathcal{R}(1,1)=0.$
			 We denote by  $\mathbb{D}_{R}=\left\lbrace z\in\mathbb{C}: |z|<R\right\rbrace$ a complex disc, where  $R$ is the radius of convergence of the series (\ref{r10}),  and by   $\mathcal{O}(\mathbb{D}_{R})$ the set of holomorphic functions defined on $\mathbb{D}_{R}.$
			Define the  $\mathcal{R}(p,q)-$ deformed numbers  by \cite{HB1}:
				\begin{equation}
				[n]_{\mathcal{R}(p,q)}:=\mathcal{R}(p^n,q^n),\quad n\in\mathbb{N}
				\end{equation}
generalizing known particular numbers  as follows:
\begin{itemize}
	\item  $q-$  Arick-Coon-Kuryskin deformation \cite{AC}
	\begin{equation*}
	\mathcal{R}(p,q):=\mathcal{R}(1,q)={\{p=1\}-q \over 1-q}
	\quad \mbox{and} \quad
	[n]_q ={1-q^n\over 1-q}.
	\end{equation*} 
	\item  $q-$ Quesne  deformation \cite{Q}
	\begin{equation*}
	\mathcal{R}(p,q):=\mathcal{R}(1,q)={\{p=1\}-q^{-1} \over 1-q}
	\quad \mbox{and}\quad
	[n]^Q_q ={1-q^{-n}\over q-1}.
	\end{equation*}
	\item $(p,q)-$ Jagannathan-Srinivasa deformation \cite{JS}
	\begin{equation*}
	\,\mathcal{R}(p,q)={p-q \over p-q}
	\quad \mbox{and}\quad
	[n]_{p,q}={p^n-q^n \over p-q}.
	\end{equation*}
	\item $(p^{-1},q)-$ Chakrabarty - Jagannathan deformation \cite{CJ}
	\begin{eqnarray*}
	\mathcal{R}(p,q)={1-p\,q \over (p^{-1}-q)p}
	\quad \mbox{and} \quad
	[n]_{p,q}={p^{-n}-q^n \over p^{-1}-q}.
	\end{eqnarray*}
	\item Hounkonnou-Ngompe  generalization of $q-$ Quesne deformation \cite{HmNe}
	\begin{eqnarray*}
	\mathcal{R}(p,q)={p\,q-1 \over (q-p^{-1})q}
	\quad \mbox{and} \quad 
	[n]^Q_{p,q}={p^{n}-q^{-n} \over q- p^{-1}}.
	\end{eqnarray*}
	\item $(p,q,\mu,\nu,g)-$ Hounkonnou-Ngompe  multi-parameter deformation \cite{Hounkonnou&Ngompe07a}
	\begin{eqnarray*}
	\,\mathcal{R}(p,q)=g(p,q){q^{\nu}\over p^{\mu}}{p\,q-1 \over (q-p^{-1})q}
	\quad \mbox{and} \quad 
	[n]^{\mu,\nu}_{p,q,g}=g(p,q){q^{\nu\,n}\over p^{\mu\,n}}{p^{n}-q^{-n} \over q- p^{-1}},
	\end{eqnarray*}
	where $0< pq<1,$ $p^{\mu}<q^{\nu-1},$ $p>1 $ and 
		$g$ is a well behaved real   non-negative function of deformation parameters $p$ and $q$ such that $g(p,q)\longrightarrow 1$ as$(p,q)\longrightarrow (1,1)$.
\end{itemize}
	Define now the
 $\mathcal{R}(p,q)-$ deformed factorials
\begin{equation}\label{s0}
[n]!_{\mathcal{R}(p,q)}:=\left \{
\begin{array}{l}
1\quad\mbox{for}\quad n=0\\
\\
\mathcal{R}(p,q)\cdots\mathcal{R}(p^n,q^n)\quad\mbox{for}\quad n\geq 1,
\end{array}
\right .
\end{equation}
and the  $\mathcal{R}(p,q)-$ deformed binomial coefficients
\begin{eqnarray}\label{bc}
\bigg[\begin{array}{c} m  \\ n\end{array} \bigg]_{\mathcal{R}(p,q)} := \frac{[m]!_{\mathcal{R}(p,q)}}{[n]!_{\mathcal{R}(p,q)}[m-n]!_{\mathcal{R}(p,q)}},\quad m,n=0,1,2,\cdots,\quad m\geq n
\end{eqnarray}
satisfying the relation
\begin{equation*}
\bigg[\begin{array}{c} m  \\ n\end{array} \bigg]_{\mathcal{R}(p,q)}=\bigg[\begin{array}{c} m  \\ m-n\end{array} \bigg]_{\mathcal{R}(p,q)},\quad m,n=0,1,2,\cdots,\quad m\geq n.
\end{equation*}
Consider the following linear operators defined on  $\mathcal{O}(\mathbb{D}_{R})$ by (see \cite{HB1} for more details):
\begin{eqnarray}
\;Q:\varPsi\longmapsto Q\varPsi(z):&=& \varPsi(qz),\\
\; P:\varPsi\longmapsto P\varPsi(z):&=& \varPsi(pz),
\end{eqnarray}
and the $\mathcal{R}(p,q)-$ derivative given by:
\begin{equation}\label{r5}
\partial_{\mathcal{R}( p,q)}:=\partial_{p,q}\frac{p-q}{P-Q}\mathcal{R}( P,Q)=\frac{p-q}{p^{P}-q^{Q}}\mathcal{R}(p^{P},q^{Q})\partial_{p,q}
\end{equation}
 extending known  derivatives as follows:
\begin{itemize}
	\item[(i)]  $q-$ Heine derivative \cite{Heine}.
	\begin{equation*}
	\mathcal{R}(p,q):=\mathcal{R}(1,q)=1
	\quad \mbox{and} \quad
	\partial_{q}\varPsi(z)=\frac{\varPsi(z)-\varPsi(qz)}{z(1-q)}.
	\end{equation*} 
	\item[(ii)]  $q-$ Quesne  derivative \cite{Q}.
	\begin{equation*}
	\mathcal{R}(p,q):=\mathcal{R}(1,q)={1-q^{-1} \over 1-q}
	\quad \mbox{and}\quad
	\partial_{q}\varPsi(z)=\frac{\varPsi(z)-\varPsi(q^{-1}z)}{z(q-1)}.
	\end{equation*}
	\item[(iii)] $(p,q)-$ Jagannathan-Srinivasa derivative \cite{JS}
	\begin{equation*}
	\mathcal{R}(p,q)=1\quad\mbox{and}\quad \partial_{p,q}\varPsi(z)=\frac{\varPsi(pz)-\varPsi(qz)}{z(p-q)}.
	\end{equation*}
	\item[(iv)] $(p^{-1},q)-$ Chakrabarty - Jagannathan derivative \cite{CJ}.
	\begin{eqnarray*}
		\mathcal{R}(p,q)={1-p\,q \over (p^{-1}-q)p}
		\quad \mbox{and} \quad
		\partial_{p^{-1},q}\varPsi(z)=\frac{\varPsi(p^{-1}z)-\varPsi(qz)}{z(p^{-1}-q)}.
	\end{eqnarray*}
	\item[(v)] Hounkonnou-Ngompe  generalization of $q-$ Quesne derivative \cite{HmNe}.
	\begin{eqnarray*}
		\mathcal{R}(p,q)={p\,q-1 \over (q-p^{-1})q}
		\quad \mbox{and} \quad 
		\partial^Q_{p,q}\varPsi(z)=\frac{\varPsi(pz)-\varPsi(q^{-1}z)}{z(q-p^{-1})}.
	\end{eqnarray*}
	\item[(vi)] $(p,q,\mu,\nu,g)-$ Hounkonnou-Ngompe  multi-parameter derivative \cite{Hounkonnou&Ngompe07a}.
	\begin{eqnarray*}
		\,\mathcal{R}(p,q)=g(p,q){q^{\nu}\over p^{\mu}}{p\,q-1 \over (q-p^{-1})q}
		\quad \mbox{and} \quad 
		\partial^{\mu,\nu}_{p,q}\varPsi(z)=g(p,q)\frac{\varPsi\Big({zq^{\nu}\over p^{\mu-1}}\Big)-\varPsi\Big({zq^{\nu -1}\over p^\mu}\Big)}{z(q-p^{-1})}.
	\end{eqnarray*}
\end{itemize}
The  algebra associated with the $\mathcal{R}(p,q)-$ deformation is a quantum algebra, denoted $\mathcal{A}_{\mathcal{R}(p,q)},$ generated by the set of operators $\{1, A, A^{\dagger}, N\}$ satisfying the following commutation relations:
\begin{eqnarray}
&& \label{algN1}
\quad A A^\dag= [N+1]_{\mathcal {R}(p,q)},\quad\quad\quad A^\dag  A = [N]_{\mathcal {R}(p,q)}.
\cr&&\left[N,\; A\right] = - A, \qquad\qquad\quad \left[N,\;A^\dag\right] = A^\dag
\end{eqnarray}
with its realization on  ${\mathcal O}(\mathbb{D}_R)$ given by:
\begin{eqnarray}\label{algNa}
A^{\dagger} := z,\qquad A:=\partial_{\mathcal {R}(p,q)}, \qquad N:= z\partial_z,
\end{eqnarray} 
where $\partial_z:=\frac{\partial}{\partial z}$ is the usual derivative on $\mathbb{C}.$
\subsection{Hounkonnou-Ngompe generalization of $q-$ Quesne algebra}
It corresponds to $\mathcal{R}(p,q)={p-q^{-1}\over q-p^{-1}}$, with generalized $q-$Quesne number and factorial \cite{HmNe}:
	\begin{equation}
	[n]^Q_{p,q}={p^n-q^{-n}\over q-p^{-1}},\quad [n]!^Q_{p,q}:=\left \{
	\begin{array}{l}
	1\quad\mbox{for}\quad n=0\\
	\\
	
	[1]^Q_{p,q}\cdots[n]^Q_{p,q}\quad\mbox{for}\quad n\geq 1,
	\end{array}
	\right .
	\end{equation}
and  following properties:
	\begin{eqnarray*}
	\;[-y]_{p,q}^Q&=& -p^{-y}q^y[y]_{p,q}^Q,\label{Qeq1}\\
	\;[x+y]_{p,q}^Q&=& q^{-y}[x]_{p,q}^Q+p^x[y]_{p,q}^Q= p^y[x]_{p,q}^Q+q^{-x}[y]_{p,q}^Q,\label{Qeq2}\\
	\;[x-y]_{p,q}^Q&=& q^{y}[x]_{p,q}^Q-p^{x-y}q^y[y]_{p,q}^Q= p^{-y}[x]_{p,q}^Q+p^{-y}q^{y-x}[y]_{p,q}^Q,\label{Qeq3}\\
	\;[x]_{p,q}^Q &=& \frac{q-p^{-1}}{p-q^{-1}}[2]_{p,q}^Q[x-1]_{p,q}^Q-pq^{-1}[x-2]_{p,q}^Q.\label{Qeq4}
	\end{eqnarray*}
The  generalized $q-$Quesne binomial coefficient
	\begin{eqnarray}
	\bigg[\begin{array}{c} n \\ \kappa \end{array}\bigg]_{p,q}^Q=
	\frac{[n]^Q_{p,q}!}{[\kappa]^Q_{p,q}![n-\kappa]^Q_{p,q}!},
	\quad 0\leq \kappa \leq n \in\mathbb{N},\label{Qeq6}
	\end{eqnarray}
satisfies the relations
\begin{small}
	\begin{eqnarray*}
	\bigg[\begin{array}{c} n \\ k \end{array}\bigg]_{p,q}^Q
	&=& \bigg[\begin{array}{c} n \\ n-k \end{array}\bigg]_{p,q}^Q=
	p^{k(n-k)}\bigg[\begin{array}{c} n \\ k \end{array}\bigg]_{1/qp}=
	p^{k(n-k)}\bigg[\begin{array}{c} n \\ n-k \end{array}\bigg]_{1/qp}\label{Qeq7}\\ 
	\;\left[\begin{array}{c} n+1 \\ k \end{array}\right]_{p,q}^Q &=&
	p^k\bigg[\begin{array}{c} n \\ k \end{array}\bigg]_{p,q}^Q
	+q^{-n-1+k}\bigg[\begin{array}{c} n \\ k-1 \end{array}\bigg]_{p,q}^Q,\label{Qeq8}\\
	\;\bigg[\begin{array}{c} n+1 \\ k \end{array}\bigg]_{p,q}^Q &=&
	p^{k}\bigg[\begin{array}{c} n \\ k \end{array}\bigg]_{p,q}^Q +
	p^{n+1-k}\bigg[\begin{array}{c} n \\ k-1 \end{array}\bigg]_{p,q}^Q
	-(p^n-q^{-n})
	\bigg[\begin{array}{c} n-1 \\ k-1 \end{array}\bigg]_{p,q}^Q.
	\end{eqnarray*}	
\end{small}
	Finally, the algebra $\mathcal{A}^Q_{p,q}$,
	generated by the set $\{1, A, A^{\dagger}, N \}$,
	associated with the  generalized $q-$Quesne deformation, satisfies the following commutation relations:
	\begin{eqnarray}
	p^{-1}A\;A^\dag- A^\dag A= q^{-N-1}, \quad&& qA\;A^\dag- A^\dag A= p^{N+1}\cr
	[N,\;A^\dag]= A^\dag,\qquad\qquad\qquad&& [N,\;A]= -A.\label{Qalg}
	\end{eqnarray}
	 We also have the following identities used in the sequel:
	\begin{equation}
	[n]_{p,q^{-1}}= p^{-1}\,q[n]^Q_{p,q} \quad  \mbox{and}\quad \bigg[\begin{array}{c}
	n \atop \\ \kappa
	\end{array}\bigg]_{p,q^{-1}}=\bigg[\begin{array}{c}
	n \atop \\ \kappa
	\end{array}\bigg]^Q_{p,q},
	\end{equation} 
	\begin{equation}
	[n]^{\mu,\nu}_{p,q,g}=g(p,q){q^{\nu\,n}\over p^{\mu\,n}}[n]^Q_{p,q}\quad \mbox{and}\quad \bigg[\begin{array}{c}
	n \atop \\ \kappa
	\end{array}\bigg]^{\mu,\nu}_{p,q,g}={q^{\nu\kappa(n-\kappa)}\over p^{\mu\kappa(n-\kappa)}}\bigg[\begin{array}{c}
	n \atop \\ \kappa
	\end{array}\bigg]^Q_{p,q},
	\end{equation}
	where $\mu,$ $\nu,$ $p,$ and $q$ are such that
	$0< pq<1,$ $p^{\mu}<q^{\nu-1},$ $p>1.$
 $g$ is a well behaved real   non-negative function of deformation parameters $p$ and $q$ such that $g(p,q)\longrightarrow 1$ as $(p,q)\longrightarrow (1,1)$.
	\section{Main results}
	\subsection{$\mathcal{R}(p,q)-$ deformed combinatorics}
	Given the above results,  we now  introduce functions $\epsilon_i >0,$ with $i=1,2,$  characterizing  particular deformations as follows:
	\begin{itemize}
		\item[(i)] $q-$ Arick-Coon-Kuryskin  deformation
		\begin{equation*}
		\epsilon_1=1 \quad \mbox{and} \quad \epsilon_2=q.
		\end{equation*}
		\item[(ii)]  $q-$ Quesne deformation
		\begin{equation*}
		\epsilon_1=1 \quad \mbox{and} \quad \epsilon_2=q^{-1}.
		\end{equation*}
		\item[(iii)] $(p,q)-$ Jagannathan-Srinivasa deformation
		\begin{equation*}
		\epsilon_1=p \quad \mbox{and} \quad \epsilon_2=q.
		\end{equation*}
	\item[(iv)] $(p^{-1},q)-$ Chakrabarty -Jagannathan deformation
	\begin{equation*}
	\epsilon_1=p^{-1} \quad \mbox{and} \quad \epsilon_2=q.
	\end{equation*}
	\item[(v)] Hounkonnou-Ngompe  generalization of $q-$ Quesne deformation
	\begin{equation*}
	\epsilon_1=p \quad \mbox{and} \quad \epsilon_2=q^{-1}.
	\end{equation*}
	\item[(vi)] $(p,q,\mu,\nu,g)-$ Hounkonnou-Ngompe   multi-parameter deformation
	\begin{equation*}
	\epsilon_1=p \quad \mbox{and} \quad \epsilon_2=q^{-1}.
	\end{equation*}
\end{itemize}
	The  $j^{th}-$ order factorial of the $\mathcal{R}(p,q)-$ number  is given by : 
	\begin{equation}\label{a1}
	[n]_{j,\mathcal{R}(p,q)}
	:= \prod_{v=0}^{j-1}\mathcal{R}(p^{n-v},q^{n-v}),\quad 
	\end{equation}
	where $n\in\mathbb{N}$ and $j\in\mathbb{N}\backslash\{0\}.$
	
	The  $\mathcal{R}(p,q)-$ deformed shifted factorial is expressed as follows:
	\begin{equation}\label{a}
	\big(x \oplus y\big)^n_{\mathcal{R}(p,q)}: = \displaystyle \prod_{i=1}^{n}\big(x\,\epsilon^{i-1}_1 + y\,\epsilon^{i-1}_2\big),\quad \mbox{with}\quad  \big(x \oplus y\big)^0_{\mathcal{R}(p,q)}: =1, 
	\end{equation}
	where  $x,y\in\mathbb{N},$ while
	the $\mathcal{R}(p,q)-$ deformed  Euler formula is developed as:
	\begin{equation}\label{b}
	\big(x \oplus y\big)^n:= \sum_{\kappa=0}^{n}\bigg[\begin{array}{c}
	n \atop \\ \kappa
	\end{array}\bigg]_{\mathcal{R}(p,q)}\,\epsilon_1^{{n-\kappa \choose 2}}\,\epsilon_2^{\kappa \choose 2}\,x^{n-\kappa}\,y^\kappa,\quad x,y\in\mathbb{N}.
	\end{equation}
	  The case of $q-$factorial in  \cite{CA1,K2}
			 corresponds to $\mathcal{R}(p,q)=q,$  $\epsilon_1=1$ and $\epsilon_2=q$ leading to :
			 \begin{itemize}
			 	\item 
				The  $j^{th}-$ order factorial of the $q-$ number:
				\begin{equation*}
				[n]_{j,q}
				:= \prod_{v=0}^{j-1}[n-v]_q,\quad  n\in\mathbb{N}\quad \mbox{and}\quad j\in\mathbb{N}\backslash\{0\},
				\end{equation*}
				\item  The  $q-$ deformed shifted factorial:
				\begin{equation*}
				\big(x \oplus y\big)^n_{q}: = \displaystyle \prod_{i=1}^{n}\big(x + y\,q^{i-1}\big),\quad \mbox{with}\quad  \big(x \oplus y\big)^0_{q}: =1, 
				\end{equation*}
				where  $x,y\in\mathbb{N}.$
				\item The $q-$ deformed  Euler formula:
				\begin{equation*}
				\big(x \oplus y\big)_q^n:= \sum_{\kappa=0}^{n}\bigg[\begin{array}{c}
				n \atop \\ \kappa
				\end{array}\bigg]_{q}\,q^{\kappa \choose 2}\,x^{n-\kappa}\,y^\kappa,\quad x,y\in\mathbb{N}.
				\end{equation*}
			\end{itemize}
				\begin{theorem}
					The $\mathcal{R}(p,q)-$ deformed Vandermonde's formula is given by:
					\begin{small}
					\begin{eqnarray}\label{vd1}
					[u+v]_{n,\mathcal{R}(p,q)}
					= \sum_{\kappa=0}^{n}\bigg[\begin{array}{c} n  \\ \kappa\end{array} \bigg]_{\mathcal{R}(p,q)}\epsilon_1^{\kappa(v-n+\kappa)}\epsilon_2^{(n-\kappa)(u-\kappa)}[u]_{\kappa,\mathcal{R}(p,q)}[v]_{n-\kappa,\mathcal{R}(p,q)},
					\end{eqnarray}
					or, equivalently,
					\begin{eqnarray}\label{vd2}
					[u+v]_{n,\mathcal{R}(p,q)}
					= \sum_{\kappa=0}^{n}\bigg[\begin{array}{c} n  \\ \kappa\end{array} \bigg]_{\mathcal{R}(p,q)}\epsilon_1^{(n-\kappa)(u-\kappa)}\epsilon_2^{\kappa(v-n+\kappa)}[u]_{\kappa,\mathcal{R}(p,q)}[v]_{n-\kappa,\mathcal{R}(p,q)},
					\end{eqnarray}
				\end{small}
					where $n$ is a positive integer.
				\end{theorem}
		\begin{proof}
			For $n\in \mathbb{N}\backslash \{0\},$ we consider the following expression:
			\begin{small}
			\begin{eqnarray*}
			T_n(u;v)_{\mathcal{R}(p,q)}:= \sum_{\kappa=0}^{n}\bigg[\begin{array}{c} n  \\ \kappa\end{array} \bigg]_{\mathcal{R}(p,q)}\epsilon_1^{\kappa(v-n+\kappa)}\epsilon_2^{(n-\kappa)(u-\kappa)}\,[u]_{\kappa,\mathcal{R}(p,q)}[v]_{n-\kappa,\mathcal{R}(p,q)}.
			\end{eqnarray*}
			\end{small}
			For $n=1,$ we have
			$T_1(u;v)_{\mathcal{R}(p,q)}
			= [u+v]_{\mathcal{R}(p,q)}.$ Using the recursion relation 
			
			\begin{equation}\label{bc1}
			\bigg[\begin{array}{c} x  \\ \kappa \end{array} \bigg]_{\mathcal{R}(p,q)} = \epsilon^{\kappa}_1\,\bigg[\begin{array}{c} x -1 \\ \kappa \end{array} \bigg]_{\mathcal{R}(p,q)} + \epsilon^{x-\kappa}_2\,\bigg[\begin{array}{c} x-1  \\ \kappa-1 \end{array} \bigg]_{\mathcal{R}(p,q)},\quad\kappa\in\mathbb{N}
			\end{equation}
			 and 
			\begin{small}
			\begin{eqnarray*}\label{vdc}
			&&[u+v-n+1]_{\mathcal{R}(p,q)}
			\,[u]_{\kappa,\mathcal{R}(p,q)}[v]_{n-\kappa-1,\mathcal{R}(p,q)}
			=  \epsilon_2^{u-\kappa}[u]_{\kappa,\mathcal{R}(p,q)}[v]_{n-\kappa,\mathcal{R}(p,q)}\cr&&
			\qquad\qquad\qquad\qquad\qquad\qquad\qquad +
			\epsilon_1^{v-n+\kappa+1}
			[u]_{\kappa+1,\mathcal{R}(p,q)}[v]_{n-\kappa-1,\mathcal{R}(p,q)},	
			\end{eqnarray*}
			\end{small}
			we obtain
			\begin{small}
			\begin{eqnarray*}\label{spa}
				T_n(u;v)_{\mathcal{R}(p,q)} &=& \sum_{\kappa=0}^{n-1}\bigg[\begin{array}{c} n-1  \\ \kappa\end{array} \bigg]_{\mathcal{R}(p,q)}\epsilon_1^{\kappa(v-n+\kappa+1)}\epsilon_2^{(n-\kappa)(u-\kappa)}\,[u]_{\kappa,\mathcal{R}(p,q)}[v]_{n-\kappa,\mathcal{R}(p,q)}\nonumber\\
				\nonumber\\
				&+& \sum_{\kappa=0}^{n-1}\bigg[\begin{array}{c} n-1  \\ \kappa\end{array} \bigg]_{\mathcal{R}(p,q)}\epsilon_1^{(\kappa+1)(v-n+\kappa+1)}\epsilon_2^{(n-\kappa-1)(u-\kappa)}\nonumber\\
				&\times&[u]_{\kappa+1,\mathcal{R}(p,q)}[v]_{n-\kappa-1,\mathcal{R}(p,q)}\nonumber\\
				&=& \big[u+v-n+1\big]_{\mathcal{R}(p,q)}\,T_{n-1}(u;v)_{\mathcal{R}(p,q)}.
			\end{eqnarray*}
			\end{small}
			Therefore, for $n\in \mathbb{N}\backslash \{0\},$  $T_n(u;v)_{\mathcal{R}(p,q)}$ satisfies the first-order recursion relation 
			\begin{equation*}
			T_n(u;v)_{\mathcal{R}(p,q)} =  [u+v-n+1]_{\mathcal{R}(p,q)}\,T_{n-1}(u;v)_{\mathcal{R}(p,q)},
			\end{equation*}
			with $T_1(x;y)_{\mathcal{R}(p,q)} = [u+v]_{\mathcal{R}(p,q)}.$ Recursively,  
			it follows that $T_n(u;v)_{\mathcal{R}(p,q)} =
			[u+v]_{n,\mathcal{R}(p,q)}.$
			Therefore, we get  (\ref{vd1}). Finally, interchanging $x$ and $y,$ and  replacing $\kappa$ by $n-\kappa,$ the expression (\ref{vd1}) is rewritten in the form (\ref{vd2}). 
		\end{proof}	 
Taking $\mathcal{R}(p,q)=q,$   $\epsilon_1=1$ and $\epsilon_2=q$,  the $q-$ deformed Vandermonde's  formula obtained in \cite{CA2,CA1} is recovered as:
\begin{eqnarray*}
[u+v]_{n, q}
&=& \sum_{\kappa=0}^{n}\bigg[\begin{array}{c} n  \\ \kappa\end{array} \bigg]_{q} q^{(n-\kappa)(u-\kappa)}\,[u]_{\kappa, q}[v]_{n-\kappa, q}\nonumber\\
&=& \sum_{\kappa=0}^{n}\bigg[\begin{array}{c} n  \\ \kappa\end{array} \bigg]_{q} q^{\kappa(v-n+\kappa)}\,[u]_{\kappa, q}[v]_{n-\kappa, q},
\end{eqnarray*}
where $n$ is a positive integer.

For $\mathcal{R}(p,q)=((q-p^{-1})q)^{-1}(p\,q-1),$ involving $\epsilon_1=p,$   $\epsilon_2=q^{-1}$   
and $n,$  a positive integer, the generalized $q-$ Quesne Vandermonde's formula can be derived as follows:
\begin{eqnarray*}\label{vdga}
	{[u+v]^Q_{n,p,q}}
	&=& \sum_{\kappa=0}^{n}\bigg[\begin{array}{c} n  \\ \kappa\end{array} \bigg]^Q_{p,q}p^{\kappa(v-n+\kappa)}q^{-(n-\kappa)(u-\kappa)}{[u]^Q_{\kappa,p,q}\,[v]^Q_{n-\kappa,p,q}}
	\\
	&=& \sum_{\kappa=0}^{n}\bigg[\begin{array}{c} n  \\ \kappa\end{array} \bigg]^Q_{p,q}p^{(n-\kappa)(u-\kappa)}q^{-\kappa(v-n+\kappa)}{[u]^Q_{\kappa,p,q}\,[v]^Q_{n-\kappa,p,q}}.
\end{eqnarray*}
\begin{theorem} 
	The negative $\mathcal{R}(p,q)-$deformed Vandermonde's formula is given by:
	\begin{small}
	\begin{eqnarray}\label{nvd1}
	{[u+v]_{-n,\mathcal{R}(p,q)}}
	= \sum_{\kappa=0}^{\infty}\bigg[\begin{array}{c} -n  \\ \kappa\end{array} \bigg]_{\mathcal{R}(p,q)}\epsilon_1^{\kappa(v+n+\kappa)}\epsilon_2^{(-n-\kappa)(u-\kappa)}{[u]_{\kappa,\mathcal{R}(p,q)}[v]_{-n-\kappa,\mathcal{R}(p,q)}},
	\end{eqnarray}
	or
	\begin{eqnarray}\label{nvd2}
	{[u+v]_{-n,\mathcal{R}(p,q)}}
	= \sum_{\kappa=0}^{\infty}\bigg[\begin{array}{c} -n  \\ \kappa\end{array} \bigg]_{\mathcal{R}(p,q)}\epsilon_1^{(-n-\kappa)(u-\kappa)}\epsilon_2^{\kappa(v+n+\kappa)}{[u]_{\kappa,\mathcal{R}(p,q)}[v]_{-n-\kappa,\mathcal{R}(p,q)}},
	\end{eqnarray}
\end{small}
	where  
	$n$ is a positive integer.
\end{theorem}
\begin{proof}
	 For $n\in\mathbb{N}\backslash\{0\},$ we consider the following  expression: 
	 \begin{small}
\begin{eqnarray*}\label{nvd3}
H_n(u;v)_{\mathcal{R}(p,q)}= \sum_{\kappa=0}^{\infty}\bigg[\begin{array}{c} -n  \\ \kappa\end{array} \bigg]_{\mathcal{R}(p,q)}\epsilon_1^{\kappa(v+n+\kappa)}\epsilon_2^{(-n-\kappa)(u-\kappa)}{[u]_{\kappa,\mathcal{R}(p,q)}[v]_{-n-\kappa,\mathcal{R}(p,q)}}.
\end{eqnarray*}
\end{small}
For $n=1,$ we have
\begin{equation*}\label{nvd6}
H_1(u;v)_{\mathcal{R}(p,q)}= {1\over \big[u+v+1\big]_{\mathcal{R}(p,q)}}.
\end{equation*}
Using  the relation (\ref{bc1}) and  
\begin{small}
\begin{eqnarray*}\label{nvd8}
	&&[u+v+n+1]_{\mathcal{R}(p,q)}{[u]_{\kappa,\mathcal{R}(p,q)}[v]_{-n-\kappa-1,\mathcal{R}(p,q)}}
	= \epsilon_2^{u-\kappa}{[u]_{\kappa,\mathcal{R}(p,q)}[v]_{-n-\kappa,\mathcal{R}(p,q)}}
	\cr&& \qquad\qquad\qquad\qquad\qquad\qquad+ \epsilon_1^{v+n+\kappa+1}
	{[u]_{\kappa+1,\mathcal{R}(p,q)}[v]_{-n-\kappa-1,\mathcal{R}(p,q)}},
\end{eqnarray*} 
\end{small}
we get:
\begin{eqnarray*}\label{nvd7}
	H_n(u;v)_{\mathcal{R}(p,q)} &=& \sum_{\kappa=0}^{\infty}\bigg[\begin{array}{c} -n-1  \\ \kappa\end{array} \bigg]_{\mathcal{R}(p,q)}\epsilon_1^{\kappa(v+n+\kappa)+\kappa}\epsilon_2^{(-n-\kappa)(u-\kappa)}\nonumber\\&\times&{[u]_{\kappa,\mathcal{R}(p,q)}[v]_{-n-\kappa,\mathcal{R}(p,q)}}
	+ \sum_{\kappa=0}^{\infty}\bigg[\begin{array}{c} -n-1  \\ \kappa\end{array} \bigg]_{\mathcal{R}(p,q)}\epsilon_1^{\kappa(v+n+\kappa)+\kappa}\nonumber\\&\times&\epsilon_2^{(-n-\kappa)(u-\kappa+1)}{[u]_{\kappa+1,\mathcal{R}(p,q)}[v]_{-n-\kappa-1,\mathcal{R}(p,q)}}\nonumber\\
	&=& {H_{n-1}(u;v)_{\mathcal{R}(p,q)} \over [u+v+n+1]_{\mathcal{R}(p,q)}}.
\end{eqnarray*}
Therefore, for $n\in\mathbb{N}\backslash\{0\},$ the sum $H_n(u;v)_{\mathcal{R}(p,q)}$ satisfies the first-order recursion relation 
\begin{equation*}\label{nvd10}
H_n(u;v)_{\mathcal{R}(p,q)} ={H_{n-1}(u;v)_{\mathcal{R}(p,q)} \over \big[u+v+n+1\big]_{\mathcal{R}(p,q)}},\quad n\in\mathbb{N}\backslash\{0\},
\end{equation*}
with $H_1(u;v)_{\mathcal{R}(p,q)} =  {1\over [u+v+1]_{\mathcal{R}(p,q)}}.$ 
Recursively, it comes that $H_n(u;v)_{\mathcal{R}(p,q)} = {[u+v]_{-n,\mathcal{R}(p,q)}}.$ 
Following the steps used to prove (\ref{nvd1}), 
we obtain (\ref{nvd2}), and the proof is achieved. 
\end{proof}
We recover the negative $q-$ Vandermonde's formula obtained in \cite{CA1}  by taking $\mathcal{R}(p,q)=q,$ $\epsilon_1=1$ and $\epsilon_2=q$:
\begin{eqnarray*}
{[x+y]_{-n,q}}
&=& \sum_{\kappa=0}^{n}q^{-(n-\kappa)(x-\kappa)}\bigg[{-n \atop \kappa}\bigg]_q\,{[x]_{\kappa,q}\,[y]_{-n-\kappa,q}},\quad |q^{(x+y+1)}|<1\nonumber\\
&=& \sum_{\kappa=0}^{n}q^{\kappa)(y+n+\kappa)}\bigg[{-n \atop \kappa}\bigg]_q\,{[x]_{\kappa,q}\,[y]_{-n-\kappa,q}},
\,\, |q^{-(x+y+1)}|<1,
\end{eqnarray*}
where $0<q< 1.$  
\begin{lemma}
	 The following relations hold:
	\begin{equation}\label{ndv11}
	{1\over [v]_{n,\mathcal{R}(p,q)}}
	= \sum_{\kappa=0}^{\infty}\bigg[\begin{array}{c} n+ \kappa -1  \\ \kappa\end{array} \bigg]_{\mathcal{R}(p,q)}\epsilon^{n(u-\kappa)}_1\epsilon^{\kappa(v-n+1)}_2{[u]_{\kappa,\mathcal{R}(p,q)}\over [u+v]_{n+\kappa,\mathcal{R}(p,q)}}
	\end{equation}
	and
	\begin{equation}\label{ndv12}
	{1\over [v]_{n,\mathcal{R}(p,q)}}
	= \sum_{\kappa=0}^{\infty}\bigg[\begin{array}{c} n+ \kappa -1  \\ \kappa\end{array} \bigg]_{\mathcal{R}(p,q)}\epsilon^{\kappa(v-n+1)}_1\epsilon^{n(u-\kappa)}_2{[u]_{\kappa,\mathcal{R}(p,q)}\over [u+v]_{n+\kappa,\mathcal{R}(p,q)}},
	\end{equation}
	where  $u,$ $v,$ $p,$ and $q$ are  real numbers such that $ 0 < q < p \leq 1.$ 
\end{lemma}
\begin{proof}
	 For $n$ a positive integer, the $\mathcal{R}(p,q)-$ deformed factorial  of order $\kappa$ of $x=-n$ is written as:
\begin{equation*}\label{ndv13}
{[-n]_{\kappa,\mathcal{R}(p,q)}} 
= (-1)^{\kappa}\,\Big(\epsilon_1\epsilon_2\Big)^{-n\kappa- {\kappa \choose 2}}\,{[n+\kappa-1]_{\kappa,\mathcal{R}(p,q)}}
\end{equation*}
and 
\begin{equation*}\label{ndv14}
\mathcal{R}(p^{-j},q^{-j}) =- \Big(\epsilon_1\,\epsilon_2\Big)^{-j}\,\mathcal{R}(p^j,q^j)\mbox{,}\quad j\in\{n,n+1,\cdots, n+\kappa-1\}.
\end{equation*}
In the same vein, the $\mathcal{R}(p,q)-$ deformed binomial  coefficient of $x=-n$ is given by:
\begin{equation*}\label{ndv15}
\bigg[\begin{array}{c} -n  \\ \kappa \end{array} \bigg]_{\mathcal{R}(p,q)} =(-1)^{\kappa}\Big(\epsilon_1\,\epsilon_2\Big)^{-n\kappa -{\kappa \choose 2} }\,\,\bigg[\begin{array}{c} n+\kappa -1  \\ \kappa \end{array} \bigg]_{\mathcal{R}(p,q)}.
\end{equation*}
Moreover, from
\begin{eqnarray*}\label{ndv16}
	{{ [-v-1]_{-n,\mathcal{R}(p,q)}}} 
	&=& {1\over[-v-1+n]_{n,\mathcal{R}(p,q)}}
	\nonumber\\
	&=&{1 \over (-1)^n\,\Big(\epsilon_1\,\epsilon_2\Big)^{-nv + {n \choose 2}}{[v]_{n,\mathcal{R}(p,q)}}}
\end{eqnarray*} 
and 
\begin{eqnarray*}\label{ndv17}
	{[-u-v-1]_{-n-\kappa,\mathcal{R}(p,q)}} 
	={(-1)^{-n-\kappa} \over \Big(\epsilon_1\,\epsilon_2\Big)^{-(n+\kappa)(u+v) + {n+\kappa \choose 2}}{[u+v]_{n+\kappa,\mathcal{R}(p,q)}}},
\end{eqnarray*} 
the relation (\ref{nvd1}) can be  written as follows:
\begin{equation}\label{ndv18}
{1\over (\epsilon_1\,\epsilon_2)^{-nv+{n\choose 2}}{[v]_{n,\mathcal{R}(p,q)}}}
= \sum_{\kappa=0}^{\infty}\bigg[\begin{array}{c} n+\kappa -1  \\ \kappa \end{array} \bigg]_{\mathcal{R}(p,q)}\,\,F(p,q),
\end{equation} 
where
\begin{equation*}\label{ndv19}
F(p,q) = {\epsilon^{{\kappa(v+n+\kappa)}}_1(\epsilon_1\epsilon_2)^{-(n+\kappa)(u-\kappa) -n\kappa- {\kappa \choose 2}}{[u]_{\kappa,\mathcal{R}(p,q)}}
	\over (\epsilon_1\,\epsilon_2)^{-(n+\kappa)(u+v) + {n+\kappa \choose 2}}{[u+v]_{n+\kappa,\mathcal{R}(p,q)}}}.
\end{equation*}
After re-arranging, 
the relation (\ref{ndv18}) is reduced to (\ref{ndv11}). Similarly, we obtain (\ref{ndv12}). 
\end{proof}

Setting $\epsilon_1=1$ and $\epsilon_2=q$ provides the $q-$analogs of the formulae (\ref{ndv11}) and (\ref{ndv12}) of \cite{CA3} as :
\begin{eqnarray*}\label{qndv11}
{1\over [v]_{n,q}}
&=& \sum_{\kappa=0}^{\infty}\bigg[\begin{array}{c} n+ \kappa -1  \\ \kappa\end{array} \bigg]_{q}q^{\kappa(v-n+1)}{{[u]_{\kappa,q}\over [u+v]_{n+\kappa,q}}},
\quad |q^{v}|<1\nonumber\\
&=& \sum_{\kappa=0}^{\infty}\bigg[\begin{array}{c} n+ \kappa -1  \\ \kappa\end{array} \bigg]_{q}q^{n(u-\kappa)}{{[u]_{\kappa,q}\over [u+v]_{n+\kappa,q}}},
\quad |q^{-v}|<1.
\end{eqnarray*}
\begin{remark}
	It is worth noticing that the  generalized $q-$ Quesne negative Vandermonde's formula is obtained,  by taking  $\mathcal{R}(p,q)=((q-p^{-1})q)^{-1}(p\,q-1),$ $\epsilon_1=p$ and   $\epsilon_2=q^{-1},$  under the form:
	\begin{small}
	\begin{eqnarray*}\label{nvd1g}
		{[u+v]^Q_{-n,p,q}}
		&=&
		\sum_{\kappa=0}^{\infty}\bigg[\begin{array}{c} -n  \\ \kappa\end{array} \bigg]^Q_{p,q}p^{\kappa(v+n+\kappa)}q^{-(-n-\kappa)(u-\kappa)}{[u]^Q_{\kappa,p,q}\,[v]^Q_{-n-\kappa,p,q}}
		\nonumber\\
		&=& \sum_{\kappa=0}^{\infty}\bigg[\begin{array}{c} -n  \\ \kappa\end{array} \bigg]^Q_{p,q}p^{(-n-\kappa)(u-\kappa)}q^{-\kappa(v+n+\kappa)}{[u]^Q_{\kappa,p,q}\,[v]^Q_{-n-\kappa,p,q}}
	\end{eqnarray*}
\end{small}
	and
	\begin{eqnarray*}\label{ndv11g}
		{1\over [v]^Q_{n,p,q}}
		&=& \sum_{\kappa=0}^{\infty}\bigg[\begin{array}{c} n+ \kappa -1  \\ \kappa\end{array} \bigg]^Q_{p,q}\,p^{n(u-\kappa)}\,q^{-\kappa(v-n+1)}{{[u]^Q_{\kappa,p,q}\over [u+v]^Q_{n+\kappa,p,q}}}
		\nonumber\\
		& = &\sum_{\kappa=0}^{\infty}\bigg[\begin{array}{c} n+ \kappa -1  \\ \kappa\end{array} \bigg]^Q_{p,q}\,p^{\kappa(v-n+1)}\,q^{-n(u-\kappa)}{{[u]^Q_{\kappa,p,q}\over [u+v]^Q_{n+\kappa,p,q}}}.
	\end{eqnarray*}
\end{remark}
\begin{definition}
	The noncentral $\mathcal{R}(p,q)-$ Stirling numbers of the first and second kinds, $s_{\mathcal{R}(p,q)}(n,\kappa,j)$ and $S_{\mathcal{R}(p,q)}(n,\kappa,j)$ are defined via the relations:
	\begin{equation}\label{s4}
	{[x-j]_{n,\mathcal{R}(p,q)}}
	:=\epsilon_2^{-{n\choose 2}-j\,n}\,\sum_{\kappa=0}^{n}s_{\mathcal{R}(p,q)}(n,\kappa;j)\,{[x]^\kappa_{\mathcal{R}(p,q)}},
	\end{equation}
	and
	\begin{equation}\label{s5}
	{[x]^n_{\mathcal{R}(p,q)}}
	:=\sum_{\kappa=0}^{n}\epsilon_2^{{\kappa\choose 2}+j\,\kappa}\,S_{\mathcal{R}(p,q)}(n,\kappa;j)\,{[x-j]_{\kappa,\mathcal{R}(p,q)}},
	\end{equation}
	where  $n\in\mathbb{N}$ and $x$ is a real number.
\end{definition}
For $j=0$, we obtain  $s_{\mathcal{R}(p,q)}(n, \kappa;0) =s_{\mathcal{R}(p,q)}(n,\kappa)$ and $S_{\mathcal{R}(p,q)}(n,\kappa; 0) =S_{\mathcal{R}(p,q)}(n,\kappa),$ which are  the $\mathcal{R}(p,q)-$ deformed Stirling numbers of the first and second kinds, respectively.
		\begin{remark} Some remarkable particular cases deserve notification:
			\begin{itemize}
				\item 
		By setting $\mathcal{R}(p,q)=q,$   $\epsilon_1=1$ and $\epsilon_2=q,$ we recover the noncentral $q-$ Stirling numbers of the first and second kinds, $s_{q}(n,\kappa,j)$ and $S_{q}(n,\kappa,j),$ derived in  \cite{CA2, CA1}:
			\begin{equation*}
			{[x-j]_{n,q}}
			=q^{-{n\choose 2}-j\,n}\,\sum_{\kappa=0}^{n}s_{q}(n,\kappa;j)\,{[x]^\kappa_{q}},
			\end{equation*}
			and
			\begin{equation*}
			{[x]^n_{q}}
			=\sum_{\kappa=0}^{n}\,q^{{\kappa\choose 2}+j\,\kappa}\,S_{q}(n,\kappa;j)\,{[x-j]_{\kappa, q}},
			\end{equation*}
			where  $n\in\mathbb{N}$ and $x$ is a real number.
		 \item For $j = 0,$ we obtain  $s_{q}(n, \kappa;0) = s_{q}(n, \kappa)$ and $S_{q}(n, \kappa; 0) =
				S_{q}(n, \kappa),$ which are  the $q-$ deformed Stirling numbers of the first and second kinds, respectively.
				\item  Setting $\mathcal{R}(p,q)={p\,q-1\over (q-p^{-1})q}$ corresponding to  $\epsilon_1=p$ and $\epsilon_2=q^{-1},$  we deduct the noncentral generalized  $q-$  Quesne Stirling numbers of the first and second kinds, $s^Q_{p,q}(n,\kappa,j)$ and $S^Q_{p,q}(n,\kappa,j),$ as follows:
					\begin{equation*}
					{[x-j]^Q_{n,p,q}}
					=q^{-{n\choose 2}-j\,n}\,\big({p\over q}\big)^n\sum_{\kappa=0}^{n}s^Q_{p,q}(n,\kappa;j)\,\Big({q\over p}[x]^Q_{p,q}\Big)^\kappa,
					\end{equation*}
					and 
					\begin{equation*}
					\big([x]^Q_{p,q}\big)^n
					=\big({p\over q}\big)^n\sum_{\kappa=0}^{n}\,q^{{\kappa\choose 2}+j\,\kappa}\big({q\over p}\big)^\kappa\,S^Q_{p,q}(n,\kappa;j)\,{[x-j]^Q_{\kappa, p,q}},
					\end{equation*}
					where  $n\in\mathbb{N}$ and $x$ is a real number.
			\item For $j = 0$,  we obtain  $s^Q_{p,q}(n, \kappa;0) = s^Q_{p,q}(n, \kappa)$ and $S^Q_{p,q}(n, \kappa; 0) =
						S^Q_{p,q}(n, \kappa),$ which are  the generalized $q-$ Quesne Stirling numbers of the first and second kinds, respectively.
			\end{itemize}
		\end{remark}
	\subsection{$\mathcal{R}(p,q)-$ deformed factorial and binomial moments} 
 We introduce now the $\mathcal{R}(p,q)-$ deformed factorial and binomial moments. For that, we suppose  $X$ is a discrete non-negative integer valued random variable, and for $x\in \mathbb{N},$ we consider  the distribution function $g$ of the variable $X$ such that $g(x)=P(X=x).$
	\begin{definition}
		The $r^{th}-$ order $\mathcal{R}(p,q)-$ factorial moment of the random variable $X$ is given by: 
		\begin{equation}
		{\bf E}\Big([X]_{r,\mathcal{R}(p,q)}\Big) := \sum_{x=r}^{\infty}\,[x]_{r,\mathcal{R}(p,q)}\,g(x),
		\end{equation}
		while the $r^{th}-$ order $\mathcal{R}(p,q)-$ binomial moment of the random variable $X$ is provided by:
		\begin{equation}\label{fb2}
		{\bf E}\bigg(\bigg[\begin{array}{c} X  \\ r\end{array} \bigg]_{\mathcal{R}(p,q)}\bigg)  = \sum_{x=r}^{\infty}\bigg[\begin{array}{c} x  \\ r\end{array} \bigg]_{\mathcal{R}(p,q)}\,g(x),
		\end{equation}
		where $r\in \mathbb{N} \backslash \{0\}.$
	\end{definition}
\begin{definition} For   $r=1,$ we obtain:
	\begin{itemize}
		\item 
			The $\mathcal{R}(p,q)-$ mean value of the random variable $X$:
		\begin{equation}
		\mu_{\mathcal{R}(p,q)}:={\bf E}\Big(\big[X\big]_{\mathcal{R}(p,q)}\Big)=\sum_{x=1}^{\infty}\big[x\big]_{\mathcal{R}(p,q)}\,g(x).
		\end{equation}
		\item The $\mathcal{R}(p,q)-$ variance  of the random variable $X$:
		\begin{equation}
		\sigma^2_{\mathcal{R}(p,q)}:={\bf V}\Big(\big[X\big]_{\mathcal{R}(p,q)}\Big)= {\bf E}\Big(\big[X\big]^2_{\mathcal{R}(p,q)}\Big) - \bigg[{\bf E}\Big(\big[X\big]_{\mathcal{R}(p,q)}\Big)\bigg]^2.
		\end{equation}
	\end{itemize}
\end{definition}
\begin{proposition}
	 In terms of Stirling number, 
		the binomial  moment can be re-expresed  as follows:
		\begin{eqnarray}\label{fba6}
		{\bf E}\bigg[\bigg(\begin{array}{c} X  \\ j\end{array} \bigg)\bigg]= \sum_{m=j}^{\infty} (-1)^{m-j}{(\epsilon_1-\epsilon_2)^{m-j}\over \epsilon^{-{m\choose 2}+\tau(m-j)}_1}s_{\mathcal{R}(p,q)}(m,j){\bf E}\bigg(\bigg[\begin{array}{c} X  \\ m \end{array} \bigg]_{\mathcal{R}(p,q)}\bigg),
		\end{eqnarray}
and the relation between the factorial moment and its ${\mathcal{R}(p,q)}-$ deformed counterpart is given by:
		\begin{equation}\label{fm}
		{\bf E}[(X)_j]= j!\sum_{m=j}^{\infty} (-1)^{m-j}{(\epsilon_1-\epsilon_2)^{m-j}\over \epsilon^{-{m\choose 2}+\tau(m-j)}_1}s_{\mathcal{R}(p,q)}(m,j){{\bf E}\big([ X ]_{m,\mathcal{R}(p,q)}\big)\over [m]_{\mathcal{R}(p,q)}!},
		\end{equation} 
		where $j\in\mathbb{N}\backslash\{0, 1\},$ $\tau\in\mathbb{N},$ and $s_{\mathcal{R}(p,q)}$ is the $\mathcal{R}(p,q)-$ Stirling number of the first kind.
\end{proposition}
\begin{proof}
	Multiplying the relation
	\begin{equation*}
	{x \choose j} =\sum_{m=j}^{x}(-1)^{m-j}(\epsilon_1-\epsilon_2)^{m-j} \epsilon^{{m\choose 2}-\tau(m-j)}_1 s_{\mathcal{R}(p,q)}(m,j)\,\bigg[\begin{array}{c} x  \\ m \end{array} \bigg]_{\mathcal{R}(p,q)}
	\end{equation*} 
	by the probability distribution $g,$ and summing for all $x\in\mathbb{N},$ we deduce  (\ref{fba6}) from  (\ref{fb2}). Moreover, from 
	\begin{equation*}
	{\bf E}\bigg[{X\choose j}\bigg]= { {\bf E}[(X)_j]\over j!},\quad {\bf E}\bigg(\bigg[\begin{array}{c} X  \\ m \end{array} \bigg]_{\mathcal{R}(p,q)}\bigg)={{\bf E}\big([X]_{m,\mathcal{R}(p,q)}\big)\over [m]_{\mathcal{R}(p,q)}!}
	\end{equation*} 
	and the relation (\ref{fba6}), we derive (\ref{fm}).
\end{proof}
\begin{remark} Some interesting results can easily be deduced as follows:
	\begin{itemize}
		\item 
The particular case of the $q-$ deformed binomial moment in \cite{CA3}  is retrieved as follows :
	\begin{equation*}
	{\bf E}\bigg[\bigg(\begin{array}{c} X  \\ j\end{array} \bigg)\bigg]= \sum_{m=j}^{\infty} (-1)^{m-j}{(1-q)^{m-j}}s_{q}(m,j){\bf E}\bigg(\bigg[\begin{array}{c} X  \\ m \end{array} \bigg]_{q}\bigg)
	\end{equation*}
	and the related  factorial moment is linked to its $q-$ counterpart by the relation:
	\begin{equation*}\label{fba7}
	{\bf E}[(X)_j]= j!\sum_{m=j}^{\infty} (-1)^{m-j}{(1-q)^{m-j}}s_{q}(m,j){{\bf E}\big({[ X ]_{m,q}}\big)\over [m]_{q}!},
	\end{equation*}
	where $j\in\mathbb{N}\backslash\{0\},$ and $s_{q}$ is the $q-$ Stirling number of the first kind.
\item Putting $\mathcal{R}(p,q)={p\,q-1\over (q-p^{-1})q},$ $\epsilon_1=p,$ and $\epsilon_2=q^{-1},$ we deduct the binomial moment:
\begin{equation*}\label{fbg6}
{\bf E}\bigg[\bigg(\begin{array}{c} X  \\ j\end{array} \bigg)\bigg]= \sum_{m=j}^{\infty} (-1)^{m-j}{(p-q^{-1})^{m-j}\over p^{-{m\choose 2}+\tau(m-j)}}s^Q_{p,q}(m,j){\bf E}\bigg(\bigg[\begin{array}{c} X  \\ m \end{array} \bigg]^Q_{p,q}\bigg)
\end{equation*}
whose  the factorial moment   is expressed, in terms of the  generalized $q-$ Quesne 
factorial moment, by:
\begin{equation*}\label{fbg7}
{\bf E}[(X)_j]= j!\sum_{m=j}^{\infty} (-1)^{m-j}{(p-q^{-1})^{m-j}\over p^{-{m\choose 2}+\tau(m-j)}}s^Q_{p,q}(m,j){{\bf E}\big({[ X ]^Q_{m,p,q}}\big)\over [m]^Q_{p,q}!},
\end{equation*}
where $j\in\mathbb{N}\backslash\{0, 1\},$ $\tau\in\mathbb{N},$ and $s^Q_{p,q}$ is the generalized $q-$ Quesne Stirling number of the first kind.
\end{itemize}
\end{remark}
\subsection{$\mathcal{R}(p,q)-$ deformed binomial distribution}
	\begin{lemma}
		The following relation holds:
		\begin{equation}\label{b14a}
		\sum^n_{\kappa = 
			0} \bigg[{n \atop \kappa}\bigg]_{\mathcal{R}(p,q)} x^\kappa (y \ominus v)_{\mathcal{R}(p,q)}^{n - \kappa} = \sum
		^n_{\kappa = 0} \bigg[{n \atop \kappa}\bigg]_{\mathcal{R}(p,q)} y^\kappa (x \ominus v)_{\mathcal{R}(p,q)}^{n -
			\kappa}.  
		\end{equation}
		In particular, for $x=p_0$ and $y=1,$ we obtain
		\begin{equation*}\label{b6}
		\sum^n_{\kappa =0} \bigg[{n \atop \kappa}\bigg]_{\mathcal{R}(p,q)} p_0^\kappa
		(1 \ominus 
		v)_{\mathcal{R}(p,q)}^{n - \kappa} 
		= \sum^n_{\kappa =0} \bigg[{n \atop \kappa}\bigg]_{\mathcal{R}(p,q)} (p_0 
		\ominus v)_{\mathcal{R}(p,q)}^{n-\kappa}=1,\quad \forall p_0 
		\end{equation*}
		where $x,$  $y,$ $v,$ $p_0$ and $n$  are integers.
		\end{lemma}
	 The $q-$analog of Lemma in \cite{K2}
	 is deduced by setting $\mathcal{R}(p,q)=q,$  $\epsilon_1=1,$ and $\epsilon_2=q$, as follows :
		\begin{equation*}
		\sum^n_{\kappa = 
			0} \bigg[{n \atop \kappa}\bigg]_{q} x^\kappa (y \ominus v)_{q}^{n - \kappa} = \sum
		^n_{\kappa = 0} \bigg[{n \atop \kappa}\bigg]_{q} y^\kappa (x \ominus v)_{q}^{n -
			\kappa}.  
		\end{equation*}
		In particular, for $x=p_0$ and $y=1,$ we obtain
		\begin{equation*}
		\sum^n_{\kappa =0} \bigg[{n \atop \kappa}\bigg]_{q} p_0^\kappa
		(1 \ominus 
		v)_{q}^{n - \kappa} 
		= \sum^n_{\kappa =0} \bigg[{n \atop \kappa}\bigg]_{q} (p_0 
		\ominus v)_{q}^{n-\kappa}=1,\quad \forall p_0, 
		\end{equation*}
		where $x,$  $y,$ $v,$ $p_0$ and $n$  are integers.
		
		The binomial distribution comes with
			a random variable
			$X$  taking two values, $0$ and $1,$ 
the probabilities $Pr(X=1)=p_0$ and $Pr(X=0)=1-p_0,$  and by
letting $S_n=X_1 +\cdots + X_n$ be the sum of $n$ random variables $(X_i)_{i\in\{1,2,\cdots,n\}}$ obeying the binomial law.
	\begin{definition}
		The $\mathcal{R}(p,q)-$deformed binomial distribution, with parameters $n,$ $p_0,$ $p,$ and $q,$  is given,  for 
		$0\leq \kappa \leq n$ and $0<q<p\leq 1,$ by:
		\begin{equation}\label{dbd}
		P_\kappa:=P_r\big([S_n]_{\mathcal{R}(p,q)}=[\kappa]_{\mathcal{R}(p,q)}\big)= \left[\begin{array}{c} n  \\ \kappa\end{array} \right]_{\mathcal{R}(p,q)}\,p^\kappa_0\,(1\ominus p_0)_{\mathcal{R}(p,q)}^{n-\kappa}.
		\end{equation}
	\end{definition}
	For $\mathcal{R}(p,q)=q,$ $\epsilon_1=1$ and $\epsilon_2=q$,  we obtain the $q-$ binomial distribution given in \cite{K2}:
	\begin{equation*}
	P_r\big([S_n]_{q}=[\kappa]_{q}\big)= \left[\begin{array}{c} n  \\ \kappa\end{array} \right]_{q}\,p^\kappa_0\,(1\ominus p_0)_{q}^{n-\kappa}, \;\; 0\leq \kappa \leq n;\;  0<q < 1. 
	\end{equation*}
	\begin{theorem}
		 The $j^{th}-$ order $\mathcal{R}(p,q)-$ deformed factorial moment is given by:
		\begin{equation}\label{bd1}
		{\mu_{\mathcal{R}(p,q)}}\big( [S_n]_{j,\mathcal{R}(p,q)}\big) = \,[n]_{j,\mathcal{R}(p,q)}\,p^j_0\mbox{,}\quad j\in\{1,2,\cdots,n\},
		\end{equation}
		and the factorial moment is expressed by the formula:
		\begin{equation}\label{bd2}
		\mu_{\mathcal{R}(p,q)}[\big({ S_n}\big)_i] = i!\,
		\sum_{j =i}^{n}(-1)^{j-i}\bigg[\begin{array}{c} n  \\ j\end{array} \bigg]_{\mathcal{R}(p,q)}\,p_0^j\,{(\epsilon_1-\epsilon_2)^{j-i}\over \epsilon^{-{j\choose 2}+\tau(j-i)}_1}\,s_{\mathcal{R}(p,q)}(j,i),
		\end{equation}
		where $\tau\in\mathbb{N},$ $i\in\{2,\cdots,n\},$ and $s_{\mathcal{R}(p,q)}$ is the $\mathcal{R}(p,q)-$ Stirling number of the first kind.
		The recursion relation  for the $\mathcal{R}(p,q)-$ deformed binomial distributions takes the form:
		\begin{eqnarray}\label{rrb}
		P_{\kappa+1} = {[n-\kappa]_{\mathcal{R}(p,q)}\over [\kappa+1]_{\mathcal{R}(p,q)}}{p_0\over \epsilon^{n-\kappa}_1- \epsilon^{n-\kappa}_2\,p_0}\,P_\kappa,\quad\mbox{with}\quad 	P_0= \big(1\ominus p_0\big)_{\mathcal{R}(p,q)}^n.
		\end{eqnarray}
	\end{theorem}
\begin{proof}
	The $j^{th}-$ order  $\mathcal{R}(p,q)-$factorial moment of $[S_n]_{\mathcal{R}(p,q)}$ is
	\begin{equation*}\label{bd3}
	{\mu_{\mathcal{R}(p,q)}}\Big( [S_n]_{j,\mathcal{R}(p,q)}\Big) 
	=\sum_{\kappa =j}^{n}[\kappa]_{j,\mathcal{R}(p,q)}\,\left[\begin{array}{c} n  \\ \kappa\end{array} \right]_{\mathcal{R}(p,q)}\,p^\kappa_0\,(1\ominus p_0)_{\mathcal{R}(p,q)}^{n-\kappa}.
	\end{equation*}
	Using the  relation
	\begin{equation*}\label{bd4}
	[\kappa]_{j,\mathcal{R}(p,q)}\,\left[\begin{array}{c} n  \\ \kappa\end{array} \right]_{\mathcal{R}(p,q)}=[n]_{j,\mathcal{R}(p,q)}\,\left[\begin{array}{c} n-j  \\ \kappa-j\end{array} \right]_{\mathcal{R}(p,q)},
	\end{equation*}
	we have
	\begin{small}
	\begin{eqnarray*}
		{\mu_{\mathcal{R}(p,q)}}\Big([S_n]_{j,\mathcal{R}(p,q)}\Big) 
		&=& [n]_{j,\mathcal{R}(p,q)}\sum_{\kappa =j}^{n}\left[\begin{array}{c} n-j  \\ \kappa-j\end{array} \right]_{\mathcal{R}(p,q)}p^\kappa_0\,(1\ominus p_0)_{\mathcal{R}(p,q)}^{n-\kappa}\nonumber\\
		&=& [n]_{j,\mathcal{R}(p,q)}\,p^j_0\sum_{x=0}^{n}\left[\begin{array}{c} n-j  \\ x\end{array} \right]_{\mathcal{R}(p,q)}p^{x}_0(1\ominus p_0)_{\mathcal{R}(p,q)}^{n-j-x}.
	\end{eqnarray*}
	\end{small}
	Moreover, the formula (\ref{bd2}) is obtained using the relations (\ref{bd1}) and (\ref{fm}). From the relation
	\begin{equation*}
	\bigg[\begin{array}{c} n  \\ \kappa+1\end{array} \bigg]_{\mathcal{R}(p,q)}={[n-\kappa]_{\mathcal{R}(p,q)}\over [\kappa+1]_{\mathcal{R}(p,q)}}\bigg[\begin{array}{c} n  \\ \kappa\end{array} \bigg]_{\mathcal{R}(p,q)}
	\end{equation*}
	and after computation, we obtain (\ref{rrb}).
\end{proof}
\begin{corollary}
	The recursion relation for the $q-$ deformed binomial distributions is given by:
	\begin{eqnarray}
	P_{\kappa+1} = {[n-\kappa]_{q}\over [\kappa+1]_{q}}{p_0\over 1- q^{n-\kappa}\,p_0}\,P_\kappa,\quad \mbox{with} \quad P_0=\prod_{j=1}^{n}\big(1-p_0\,q^{j-1}\big).
	\end{eqnarray}
\end{corollary}
\begin{proof}
It is obtained by straightforward deduction. 
\end{proof}
Now, let us  consider, for $v\in \mathbb{N} \backslash \{0\},$ the  $l-$order differential operator
\begin{small}
\begin{eqnarray}\label{b8}
(vD_ v)^l=\sum_{j=1}^{l}\frac{1}{[j-1]!_{\mathcal{R}(p,q)}}\Big(\sum_{t=0}^{j-1}\left[\begin{array}{c} j-1 \\ t\end{array} \right]_{\mathcal{R}(p,q)}(-1)^tq^{t \choose 2}[j-t]^{l-1}_{\mathcal{R}(p,q)}\Big)v^j(D_v)^j,
\end{eqnarray}
\end{small}
where $D_v:= {d}/{d_{\mathcal{R}(p,q)}}$ acting on $v,$ (which will be used in the sequel), 
giving, for $l=2,$ 
the following second order differential operator:
\begin{equation}\label{b9}
(v\,D_ v)^2 = \mathcal{R}(p,q)\,(v\,D_ v )+{1 \over \mathcal{R}!(p,q)}\bigg( \mathcal{R}(p^2,q^2) -  \mathcal{R}(p,q)\bigg) v^2(\,D_ v)^2 .
\end{equation}
Besides, for $\epsilon_i>0,$ with $i\in\{1,2\},$ such that $\forall p, q, \,0<q<p\leq 1,$ 
\begin{equation}\label{as1}
\mathcal{R}(p^{x-y},q^{x-y}) = \epsilon^{-y}_1\,\mathcal{R}(p^{x},q^{x}) + \epsilon^{-y}_1\epsilon^{x-y}_2\,\mathcal{R}(p^{y},q^{y}).
\end{equation}

\begin{lemma}
	The mean value of the random variable sum  $ S_n$ is given by:
	\begin{equation}\label{bmv}
	\mu_{\mathcal{R}(p,q)}\big([S_n]_{\mathcal{R}(p,q)}\big) = p_0\,\mathcal{R}(p^n,q^n).
	\end{equation}
	Its corresponding variance 
	can be written as:
	\begin{eqnarray}
	{\bf Var}\big([S_n]_{\mathcal{R}(p,q)}\big)=p_0[n]_{\mathcal{R}(p,q)}\bigg(\mathcal{R}(p,q)+ {\bf X}p_0[n-1]_{\mathcal{R}(p,q)}-p_0[n]_{\mathcal{R}(p,q)} \bigg),
	\end{eqnarray}
	where 
	${\bf X}p_0[n-1]_{\mathcal{R}(p,q)}>p_0[n]_{\mathcal{R}(p,q)}-\mathcal{R}(p,q),$
	and
	\begin{equation}
	{\bf X}=\mathcal{R}(p,q)^{-1}\big(\mathcal{R}(p^2,q^2)-\mathcal{R}(p,q)\big).
	\end{equation}
	The  mean of the product  $S_n(S_n-1)\cdots (S_n-r+1)$  is given by:
	\begin{eqnarray}
	\mu_{\mathcal{R}(p,q)}\bigg(\prod_{i=0}^{r-1}\epsilon^{-i}_2([S_n]^r_{\mathcal{R}(p,q)} - \epsilon^{r-i}_1[i]_{\mathcal{R}(p,q)})\bigg) = p^r_0\, \prod_{i=0}^{r-1}[n-i]_{\mathcal{R}(p,q)}.
	\end{eqnarray}
\end{lemma}
\begin{proof}
	Applying the $\mathcal{R}(p,q)-$derivative on $p_0 $ to the left and right hand sides of the relation (\ref{b6}) leads to
		\begin{equation*}
		p_0 D_{p_0} \sum^n_{\kappa =0} \bigg[{n \atop \kappa}\bigg]_{\mathcal{R}(p,q)} p_0^\kappa
		(1\ominus u)_{\mathcal{R}(p,q)}^{n - \kappa}
		=	\mu_{\mathcal{R}(p,q)}([S_n]_{\mathcal{R}(p,q)}),
		\end{equation*}
and
		\begin{small}
		\begin{equation*}
		p_0 D_{p_0}\sum^n_{\kappa =0} \bigg[{n \atop \kappa}\bigg]_{\mathcal{R}(p,q)} (p_0 
		\ominus u)_{\mathcal{R}(p,q)}^{n-\kappa}
		= p_0\sum^n_{\kappa =0} \bigg[{n \atop \kappa}\bigg]_{\mathcal{R}(p,q)}[n-\kappa]_{\mathcal{R}(p,q)}   (p_0 
		\ominus u)_{\mathcal{R}(p,q)}^{n-\kappa-1},
		\end{equation*}
	\end{small}
respectively.
	According to (\ref{as1}) and (\ref{bc1}), we obtain (\ref{bmv}).
	Besides,
	\begin{equation*}\label{b14}
	\mu_{\mathcal{R}(p,q)}\big([S_n]^2_{\mathcal{R}(p,q)}\big)=(p_0\,D_{p_0})^2\sum_{\kappa=0}^{n}\,\left[\begin{array}{c} n  \\ \kappa\end{array} \right]_{\mathcal{R}(p,q)}\,p^\kappa_0\,(1\ominus u)_{\mathcal{R}(p,q)}^{n-\kappa}.
	\end{equation*}
	Setting ${\bf X} =\frac{1}{\mathcal{R}!(p,q)}\bigg(\mathcal{R}(p^2,q^2)-\mathcal{R}(p,q)  \bigg)$ yields the second order differential equation
	\begin{equation*}
	(p_0\,D_{p_0})^2\sum_{\kappa=0}^{n}\,\left[\begin{array}{c} n  \\ \kappa\end{array} \right]_{\mathcal{R}(p,q)}\,(p_0\ominus u)_{\mathcal{R}(p,q)}^{n-\kappa}
	= \mathcal{R}(p,q)\mu_{\mathcal{R}(p,q)}\big([S_n]_{\mathcal{R}(p,q)}\big)+
	Y_4,
	\end{equation*}
	where
	\begin{eqnarray*}
		Y_4 
		&=&  {\bf X}\,p_0^2\sum_{\kappa=0}^{n}\mathcal{R}(p^{n-\kappa},q^{n-\kappa})\mathcal{R}(p^{n-\kappa-1},q^{n-\kappa-1})(p_0\ominus u)_{\mathcal{R}(p,q)}^{n-\kappa}.
	\end{eqnarray*}
	Using the relations (\ref{bc1}) and (\ref{b14}), we obtain
	\begin{small}
	\begin{eqnarray*}\label{e2}
	\mu_{\mathcal{R}(p,q)}\big([S_n]^2_{\mathcal{R}(p,q)}\big)= \mathcal{R}(p,q)\,\mu_{\mathcal{R}(p,q)}([S_n]_{\mathcal{R}(p,q)}) +  {\bf X}\,p_0^2\mathcal{R}(p^{n},q^{n})\mathcal{R}(p^{n-1},q^{n-1}),
	\end{eqnarray*}
\end{small}
	and from the mean value of ${ S_n},$  we deduce ${\bf Var}\big([S_n]_{\mathcal{R}(p,q)}\big) .$ Futhermore,
	applying the operator ${{p^r_0 \bigg({D_{p_0}}\bigg)^r \bigg|_{u
				= p_0}}}$ to formula  (\ref{b6}), we obtain
	
	\begin{eqnarray*}
		p^r_0 \big({D_{p_0}}\big)^r\,\sum^n_{\kappa =0} \bigg[{n \atop \kappa}\bigg] p^\kappa
		(1 \ominus 
		u)_{\mathcal{R}(p,q)}^{n - \kappa} 
		&=& \mu_{\mathcal{R}(p,q)}\big([S_n]^r_{\mathcal{R}(p,q)}\big). 
	\end{eqnarray*}
	
	Moreover,
	\begin{eqnarray*}
		p^r_0 \,\bigg({D_{p_0}}\bigg)^r(p_0 \ominus u)_{\mathcal{R}(p,q)}^{n - \kappa} &=& p^r_0\,\prod_{i=0}^{r-1}[n-\kappa-i]_{\mathcal{R}(p,q)}\,(p_0\ominus u)_{\mathcal{R}(p,q)}^{n - \kappa-r},
	\end{eqnarray*}
	or  equivalently,
	\begin{eqnarray*}
		p^r_0 \,\bigg({D_{p_0}}\bigg)^r(p_0 \ominus u)_{\mathcal{R}(p,q)}^{n - \kappa} &=&	p^r_0 \,\prod_{i=0}^{r-1}\epsilon^{-\kappa}_1[n-i]_{\mathcal{R}(p,q)}\,(p_0\ominus u)_{\mathcal{R}(p,q)}^{n - \kappa-r}\nonumber\\
		&+& p_0^r\prod_{i=0}^{r-1}\epsilon^{n-i}_2[-\kappa]_{\mathcal{R}(p,q)}(p_0 \ominus u)_{\mathcal{R}(p,q)}^{n - \kappa-r}.
	\end{eqnarray*}
	Then,
	\begin{equation*}
	p^r_0 \bigg({D_{p_0}}\bigg)^r\sum^n_{\kappa =0} \bigg[{n \atop \kappa}\bigg](p \ominus u)_{\mathcal{R}(p,q)}^{n - \kappa} 
	= p_0^r\prod_{i=0}^{r-1}[n-i]_{\mathcal{R}(p,q)}
	\end{equation*}
	gives
	\begin{equation*}
	\mu_{\mathcal{R}(p,q)}\big([S_n]^r_{\mathcal{R}(p,q)}\big) =p_0^r\prod_{i=0}^{r-1}[n-i]_{\mathcal{R}(p,q)},
	\end{equation*}
	and after computation, 
	 the result follows.
\end{proof}
The particular case of the  $q-$ deformation described in \cite{K2}
corresponds to $\mathcal{R}(p,q)=q,$  $\epsilon_1=1,$ and $\epsilon_2=q,$ and yields
the mean value of $ S_n:$
	\begin{equation*}
	\mu_{q}\big([S_n]_{q}\big) = p_0\,[n]_q,
	\end{equation*}
the variance of $ S_n:$
	\begin{eqnarray*}
	{\bf Var}\big([S_n]_{q}\big)=[n]_{q}\,p_0\,(1-p_0),
	\end{eqnarray*}
and the mean of the product  $S_n(S_n-1)\cdots (S_n-r+1):$  
	\begin{eqnarray*}
	\mu_{q}\bigg(\prod_{i=0}^{r-1}q^{-i}([S_n]^r_{q} - [i]_{q})\bigg) = p^r_0\, \prod_{i=0}^{r-1}[n-i]_{q}.
	\end{eqnarray*}
	\begin{remark}
		Deducing the above mentioned properties for the particular case of the generalized $q-$  Quesne deformation 
  leads to:
		\begin{itemize}
			\item Generalized $q-$ Quesne probability distribution:
			\begin{eqnarray*}\label{gbd}
			p_r\big(S_n=[\kappa]^Q_{p,q}\big)
			=\left[\begin{array}{c} n  \\ \kappa\end{array}
			\right]^Q_{p,q}\,p^\kappa_0\Big((1 \ominus p_0)^Q_{p,q}\Big)^{n-\kappa},\quad 0\leq \kappa \leq n;
			\end{eqnarray*}
			\item  $j^{th}-$ order  generalized $q-$ Quesne factorial moment:
			\begin{equation*}\label{bdg1}
			\mu^Q_{p,q}\big([ S_n]^Q_{j,p,q}\big) = [n]^Q_{j,p,q}\,\big({q\over p}\big)^j\,p^j_0\mbox{,}\quad j\in\{1,2,\cdots,n\};
			\end{equation*}
			\item Generalized $q-$ Quesne factorial moment: 
			\begin{equation*}\label{bdg2}
			\mu^Q_{p,q}[({ S_n})_i] = i!\,
			\sum_{j=i}^{n}(-1)^{j-i}\bigg[\begin{array}{c} n  \\ j \end{array} \bigg]^Q_{p,q}\,p_0^j\,{(p-q^{-1})^{j-i}\over p^{-{j\choose 2}+\tau(j-i)}}s^Q_{p,q}(j,i);
			\end{equation*}
			\item Recursion relation for the  generalized $q-$ Quesne distributions:
			\begin{equation*}
			P_{\kappa+1} = {[n-\kappa]^Q_{p,q}\over [\kappa+1]^Q_{p,q}}{p_0\over \big(1\ominus p_0\big)^Q_{p,q}}\,P_\kappa,\quad\mbox{with}\,P_0= \Big(\big(1\ominus p_0\big)^Q_{p,q}\Big)^n;
			\end{equation*}
			\item Mean value:
			\begin{equation*}
			\mu^Q_{p,q}(S_n)=p_0\,{q\over p}\,{p^n-q^{-n}\over q-p^{-1}};
			\end{equation*}
			\item Variance:
			\begin{eqnarray*}
			\big(\sigma^Q_{p,q}\big)^2(S_n)=p_0{q\over p}[n]^Q_{p,q}\Big(1+(p^{-1}+q-1)\,p_0{q\over p}[n-1]^Q_{p,q}-p_0{q\over p}[n]^Q_{p,q}\Big);
			\end{eqnarray*}
			\item Mean value of the product   $S_n(S_n-1)\cdots (S_n-r+1):$ 
			\begin{eqnarray*}
			{\mu}^Q_{p,q}\big(\prod_{i=0}^{r-1}p^{-i} \big([S_n]^Q_{p,q}\big)^r - q^{-r+i+1}p^{-1}[i]^Q_{p,q}\big) =q^r\, p^{-r} p^r_0\, \prod_{i=0}^{r-1}[n-i]^Q_{p,q}.
			\end{eqnarray*}
		\end{itemize}
	\end{remark}
\subsection{$\mathcal{R}(p,q)-$ deformed Euler distribution}
\begin{definition}
	The $\mathcal{R}(p,q)-$ deformed exponential functions, denoted  $E_{\mathcal{R}(p,q)}$ and  $e_{\mathcal{R}(p,q)},$ are defined as follows:
	\begin{eqnarray}
	E_{\mathcal{R}(p,q)}(z):=\sum_{n=0}^{\infty}{\epsilon^{n\choose 2}_2\,z^n\over \mathcal{R}!(p^n,q^n)}\quad \mbox{and}\quad e_{\mathcal{R}(p,q)}(z):=\sum_{n=0}^{\infty}{\epsilon^{n\choose 2}_1\,z^n\over \mathcal{R}!(p^n,q^n)}
	\end{eqnarray}
	with
	$E_{\mathcal{R}(p,q)}(-z)\,e_{\mathcal{R}(p,q)}(z)=1.$
\end{definition}
In the particular case where $\mathcal{R}(p,q)=1,$ $\epsilon_1=p,$ and $\epsilon_2=q,$ they provide  the Jagannathan-Srinivasa $q-$exponential functions  \cite{HB1}:
\begin{eqnarray*}
E_{p,q}(z)=\sum_{n=0}^{\infty}{q^{n\choose 2}\,z^n\over [n]_{p,q}!}\quad \mbox{and}\quad e_{p,q}(z)=\sum_{n=0}^{\infty}{p^{n\choose 2}\,z^n\over [n]_{p,q}!}
\end{eqnarray*}
with 
$E_{p,q}(-z)\,e_{p,q}(z)=1.$
\begin{definition}
	The  $\mathcal{R}(p,q)-$ deformed Euler distribution, with parameters $n,$ $\theta,$ $p,$ and $q,$  is defined by:
	\begin{equation}\label{ed1}
	P_r( X=x)=E_{\mathcal{R}(p,q)}(-\theta)\,\, {\epsilon^{x\choose 2}_1\theta^x\over \mathcal{R}!(p^x,q^x)}\mbox{,}\quad x\in\mathbb{N},
	\end{equation}
	where $0<q<p\leq 1,$ $0<\theta<1$
	and $\displaystyle\sum_{x=0}^{\infty}P_r( X=x)=1.$
\end{definition}
For $\mathcal{R}(p,q)=q,$ involving  $\epsilon_1=1,$ and $\epsilon_2=q$, we retrieve the $q-$ Euler distribution  obtained in \cite{CA1}:
\begin{equation}
P_r( X=x)=E_{q}(-\theta)\,\, {\theta^x\over [x]_q!},
\end{equation}
where $x\in\mathbb{N},$ $0<q <1$ and $0<\theta <{1\over 1-q}.$
\begin{theorem} 
 The $j^{th}-$ order   $\mathcal{R}(p,q)-$ factorial moment of $ X$ is given by
		\begin{equation}\label{ed2}
		\mu_{\mathcal{R}(p,q)}\big( [X]_{j,\mathcal{R}(p,q)}\big) = \theta^j\,\epsilon^{j\choose 2}_1\,E_{\mathcal{R}(p,q)}(-\theta)\,e_{\mathcal{R}(p,q)}(\epsilon^j_1\theta),\quad j\in \mathbb{N} \backslash \{0\}, 
		\end{equation}
while the factorial moment is written as:
		\begin{equation}\label{ed3}
		\mu_{\mathcal{R}(p,q)}( X)_i=i!\,\sum_{j=i}^{\infty} (-1)^{j-i}\,EU(j,i)\,s_{\mathcal{R}(p,q)}(j,i),
		\end{equation} 
		where \begin{equation}
		EU(j,i)={\theta^j\epsilon^{j\choose 2}_1\,E_{\mathcal{R}(p,q)}(-\theta)\,e_{\mathcal{R}(p,q)}(\epsilon^j_1\theta)\over [j]_{\mathcal{R}(p,q)}!} \,{(\epsilon_1 -\epsilon_2)^{j-i}\over \epsilon^{-{j\choose 2}+\tau(j-i)}_1},
		\end{equation}  $\tau\in\mathbb{N},$ $i\in\mathbb{N}\backslash\{0,1\},$ and $s_{\mathcal{R}(p,q)}$ is the $\mathcal{R}(p,q)-$ deformed Stirling number of the first kind.
 The recursion relation for the associated  $\mathcal{R}(p,q)-$ deformed Euler distributions takes the form:
		\begin{equation}\label{edr1}
		P_{x+1}= {\theta\,\epsilon^x_1 \over \mathcal{R}(p^{x+1},q^{x+1})}\,P_x,\quad \mbox{with}\, \,	P_0=E_{\mathcal{R}(p,q)}(-\theta).
		\end{equation}
\end{theorem}
\begin{proof}
From the $j^{th}-$ order $\mathcal{R}(p,q)-$ factorial moment, we get
\begin{eqnarray*}
\mu( [X]_{j,\mathcal{R}(p,q)})&=&\sum_{x=j}^{\infty}[x]_{j,\mathcal{R}(p,q)}\,E_{\mathcal{R}(p,q)}(-\theta)\,\, {\epsilon^{x\choose 2}_1\theta^x\over \mathcal{R}!(p^x,q^x)}.
\end{eqnarray*}
Using the relation
\begin{equation*}
\mathcal{R}(p^x,q^x)^j\, \mathcal{R}!(p^{x-j},q^{x-j})= \mathcal{R}!(p^x,q^x),
\end{equation*}
we obtain
\begin{eqnarray*}
\mu( [X]_{j,\mathcal{R}(p,q)})&=& E_{\mathcal{R}(p,q)}(-\theta)\,\sum_{x=j}^{\infty}\,{\epsilon^{x\choose 2}_1\theta^x\over \mathcal{R}!(p^{x-j},q^{x-j})}\nonumber\\
&=&\theta^j\,\epsilon^{j\choose 2}_1\,E_{\mathcal{R}(p,q)}(-\theta)\,\sum_{h=0}^{\infty}\,{\epsilon^{h\choose 2}_1\big(\epsilon^j_1\theta\big)^h\over [h]_{\mathcal{R}(p,q)}!}\nonumber\\
&=& \theta^j\,\epsilon^{j\choose 2}_1\,E_{\mathcal{R}(p,q)}(-\theta)\,e_{\mathcal{R}(p,q)}(\epsilon^j_1\theta).
\end{eqnarray*}
Furthermore, exploiting the relation (\ref{ed2}) and (\ref{fm}), we obtain (\ref{ed3}). Using the $\mathcal{R}(p,q)-$ factorials,  and after computation,
we get (\ref{edr1}).
	\end{proof}
\begin{corollary}
	The recursion relation for the $q-$ deformed Euler distributions is given by:
	\begin{equation}\label{edrq1}
	P_{x+1}= {\theta \over [x+1]_q}\,P_x,\quad \mbox{with}\, 	P_0=E_{q}(-\theta).
	\end{equation}
\end{corollary}
\begin{proof}
	It stems from a straightforward computation.
	\end{proof}
	Particular results for the  $q-$ deformation performed in \cite{CA1},   recovered with  $\mathcal{R}(p,q)=q,$ $\epsilon_1=1,$ and $\epsilon_2=q,$ yield
the $j^{th}-$ order  $q-$ factorial moment of $ X$:
\begin{equation*}
\mu_{q}\big( [X]_{j,q}\big) = \theta^j,\quad j\in \mathbb{N} \backslash \{0\},
\end{equation*}
and the $q-$ factorial moments:
\begin{equation*}
\mu_{q}( X)_i=i!\,\sum_{j=i}^{\infty} (-1)^{j-i}\,{\theta^j\over [j]_{q}!} \,(1 -q)^{j-i}\,s_{q}(j,i),
\end{equation*}  
 $i\in\mathbb{N}\backslash\{0\},$ and $s_{q}$ is the $q-$ deformed Stirling number of the first kind.
\begin{remark}
	The generalized $q-$ Quesne Euler distribution corresponding to the choice $\mathcal{R}(p,q)=((q-p^{-1})q)^{-1}(p\,q-1),$  $\epsilon_1=p$ and $\epsilon_2=q^{-1},$ gives the following realization:
	\begin{itemize}
		\item Generalized $q-$ Quesne  probability distribution:
		\begin{equation*}
		P_r( X=x)=E^Q_{p,q}(-\theta)\,\, {\theta^x\,p^Xq^{{x\choose 2}}\over q^x [x]^Q_{p,q}!}\mbox{,}\quad x\in \mathbb{N};
		\end{equation*}
		\item
		 $j^{th}-$ order  generalized $q-$ Quesne factorial moment:
		\begin{equation*}\label{edg2}
		\mu^Q_{p,q}\big([X]^Q_{j,p,q}\big) = \theta^j\,p^{{j\choose 2}}E^Q_{p,q}(-\theta)e^Q_{p,q}(p^j\,\theta),\quad j\in\mathbb{N}\backslash\{0\};
				\end{equation*}
		\item Generalized $q-$ Quesne factorial moment: 
		\begin{equation*}
		\mu^Q_{p,q}[(X)_i]=i!\,\sum_{j=i}^{\infty} (-1)^{j-i}\,{(p -q^{-1})^{j-i}\over p^{-{j\choose 2}+\tau(j-i)}}EU(j,i)\,s^Q_{p,q}(j,i),
		\end{equation*} 
		where:
		\begin{equation*}
		EU(j,i)={\theta^j\,p^{{j\choose 2}}E_{\mathcal{R}(p,q)}(-\theta)e^Q_{p,q}(p^j\,\theta)\over[j]^Q_{p,q}!},
		\end{equation*}
		$\tau\in\mathbb{N},$  $i\in \mathbb{N} \backslash \{0,1\},$ and $s^Q_{p,q}$ is the  generalized $q-$ Quesne  Stirling number of the first kind; 
		\item  Recursion relation for the  generalized $q-$ Quesne distributions:
		\begin{equation*}
		P_{x+1}= {\theta\,p^{x+1}\over q\, [x+1]^Q_{p,q}}\,P_x,\quad \mbox{with}\quad 	P_0=E^Q_{p,q}(-\theta).
		\end{equation*}
	\end{itemize}
\end{remark}
\subsection{$\mathcal{R}(p,q)-$ deformed P\'olya distribution}
 We here assume that  boxes are successively drawn one after the other from an urn, initially containing $r$ white and $s$ black boxes.
After each drawing, the drawn box is placed back in the urn together with $x$ boxes of the same color.
 We suppose that the probability of drawing a white box at the $i^{th}$ drawing, given that $j-1$ white boxes are drawn in the previous $i-1$ drawings, is given as:
\begin{equation*}
P_{i,j}={[r+x(j-1)]_{\mathcal{R}(p,q)}\over [r+s+x(i-1)]_{\mathcal{R}(p,q)}}={[m-j+1]_{\mathcal{R}(p^{-x},q^{-x})}\over [m+u-i+1]_{\mathcal{R}(p^{-x},q^{-x})}},
\end{equation*}
where $j\in\{1,2,\cdots,i\},$ $ i\in\mathbb{N}\backslash\{0\},$ $0<q<p\leq 1,$  $m={-r/ x},$  $u={-s/ x},$ and $x$ is a non zero integer.
We call this  model  the $\mathcal{R}(p,q)-$ P\'olya urn model.
Setting $\mathcal{R}(p,q)=q,$ $\epsilon_1=1$, and $\epsilon_2=q$,  we obtain the probability for  $q-$  P\'olya urn model described in \cite{CA2}:
\begin{equation*}
P_{i,j}={[r+x(j-1)]_{q}\over [r+s+x(i-1)]_{q}}={[m-j+1]_{q^{-x}}\over [m+u-i+1]_{q^{-x}}}.
\end{equation*}
 Let $T_n$ be the number of white boxes drawn in $n$ drawings.
Then, we have:
\begin{definition}
	The $\mathcal{R}(p,q)-$ deformed P\'olya distribution,  with parameters $m,$ $u,$ $n,$ $p,$ and $q,$ is defined by:
	\begin{small}
	\begin{eqnarray}\label{po1}
	P_\kappa:=P_r( T_n=\kappa)=\Psi(p,q)\bigg[{n \atop \kappa}\bigg]_{\mathcal{R}(p^{-x},q^{-x})}{[m]_{\kappa,\mathcal{R}(p^{-x},q^{-x})}[u]_{n-\kappa,\mathcal{R}(p^{-x},q^{-x})} \over [m+u]_{n,\mathcal{R}(p^{-x},q^{-x})}},
	\end{eqnarray}
\end{small}
	where  $0<q<p\leq 1,$ $\kappa\in\{0,\cdots,n\},$  $\Psi(p,q)={\epsilon^{-x\kappa(u-n+\kappa)}_1\, \epsilon^{-x(n-\kappa)(m-\kappa)}_2},$ $x$ is an integer  and $\displaystyle\sum_{\kappa=0}^{n}P_r( T_n=\kappa)=1.$
\end{definition}
For $\mathcal{R}(p,q)=q,$   $\epsilon_1=1,$ and $\epsilon_2=q,$
we recover
the $q-$ P\'olya distribution \cite{CA2}:
\begin{equation*}\label{qp}
P_r(T_n=\kappa) = q^{-x(n-\kappa)(m-\kappa)}\bigg[{n \atop \kappa}\bigg]_{q^{-x}}{[m]_{\kappa,q^{-x}}[u]_{n-\kappa, q^{-x}} \over [m+u]_{n,q^{-x}}},\quad x\in\mathbb{N}
\end{equation*}
where $0<q<1$ and $\kappa\in\{0,\cdots,n\}$.
\begin{theorem}
		The $j^{th}-$ order $\mathcal{R}(p,q)-$ deformed factorial moment is given by:
		\begin{small}
		\begin{eqnarray}\label{po3}
		\mu_{\mathcal{R}(p,q)}\big([T_n]_{j,\mathcal{R}(p^{-x},q^{-x})}\big) = {[n]_{j,\mathcal{R}(p^{-x},q^{-x})}\,[m]_{j,\mathcal{R}(p^{-x},q^{-x})}\over [m+u]_{j,\mathcal{R}(p^{-x},q^{-x})}},\, j\in\{1,\cdots,n\},
		\end{eqnarray}
		\end{small}
		while the  factorial moment  is expressed by:
		\begin{eqnarray}\label{po4}
		\mu_{\mathcal{R}(p,q)}[(T_n)_i] = i!\sum_{j=i}^{n}(-1)^{j-i}\bigg[\begin{array}{c} n  \\ j \end{array} \bigg]_{\mathcal{R}(p,q)}{ s_{\mathcal{R}(p^{-x},q^{-x})}(j,i)\over (\epsilon^{-x}_1-\epsilon^{-x}_2)^{i-j}}P(j,i),
		\end{eqnarray}
		where \begin{equation}
		P(j,i)={\epsilon^{j\choose 2}_1\,[m]_{j,\mathcal{R}(p^{-x},q^{-x})}\over \epsilon^{\tau(j-i)}_1\,[m+u]_{j,\mathcal{R}(p^{-x},q^{-x})}},
		\end{equation}    
		$\tau\in\mathbb{N},$ $ i\in\{2,\cdots,n\}$ and  $s_{\mathcal{R}(p^{-x},q^{-x})}$ is the $\mathcal{R}(p,q)-$ deformed Stirling number of the first kind.	
		 The recursion relation for the $\mathcal{R}(p,q)-$ P\'olya distributions is given as follows:
		\begin{equation}
		P_{\kappa +1} ={ \epsilon^{x(n+m-2\kappa-1)}_2\over \epsilon^{x(u-n+2\kappa+1)}_1} {[n-\kappa]_{\mathcal{R}(p^{-x},q^{-x})}\over [u-n+\kappa+1]_{\mathcal{R}(p^{-x},q^{-x})}}\,{[m-\kappa]_{\mathcal{R}(p^{-x},q^{-x})}\over [\kappa+1]_{\mathcal{R}(p^{-x},q^{-x})}}P_\kappa,
		\end{equation}
		with
		\begin{equation*}
		P_0= \epsilon^{-x\,m\,n}_2{[u]_{n,\mathcal{R}(p^{-x},q^{-x})} \over [m+u]_{n,\mathcal{R}(p^{-x},q^{-x})}}.
		\end{equation*}
\end{theorem}
\begin{proof}
	 The $j^{th}-$ order $\mathcal{R}(p,q)-$ factorial moment is furnished by the formula
\begin{eqnarray*}
&&\mu_{\mathcal{R}(p,q)}\big([T_n]_{j,\mathcal{R}(p^{-x},q^{-x})}\big) =\sum_{\kappa=j}^{n}[\kappa]_{j,\mathcal{R}(p^{-x},q^{-x})}\Psi(p,q)\bigg[{n \atop \kappa}\bigg]_{\mathcal{R}(p^{-x},q^{-x})}\cr&& \qquad\qquad\qquad\qquad\qquad\qquad \times {[m]_{\kappa,\mathcal{R}(p^{-x},q^{-x})}[u]_{n-\kappa,\mathcal{R}(p^{-x},q^{-x})} \over [m+u]_{n,\mathcal{R}(p^{-x},q^{-x})}}.
\end{eqnarray*}
Using the $\mathcal{R}(p,q)-$ Vandermonde's formula (\ref{vd1}) yields the expression
\begin{equation*}
\mu_{\mathcal{R}(p,q)}\big([T_n]_{j,\mathcal{R}(p^{-x},q^{-x})}\big) ={[m]_{j,\mathcal{R}(p^{-x},q^{-x})}[n]_{j,\mathcal{R}(p^{-x},q^{-x})} \over [m+u]_{n,\mathcal{R}(p^{-x},q^{-x})}}[m+u-j]_{n-j,\mathcal{R}(p^{-x},q^{-x})},
\end{equation*}
for $j\in\{1,\cdots,n\}.$ 
Then, since
\begin{eqnarray*}
[m+u]_{n,\mathcal{R}(p^{-x},q^{-x})}=[m+u]_{j,\mathcal{R}(p^{-x},q^{-x})}[m+u-j]_{n-j,\mathcal{R}(p^{-x},q^{-x})},
\end{eqnarray*}
the relation (\ref{po3}) is obtained. From  (\ref{fm}), with  $[-x]_{\mathcal{R}(p,q)}$  instead of $\mathcal{R}(p,q),$ (\ref{po4}) is deduced. Since \begin{equation*}
[\kappa+1]_{\mathcal{R}(p^{-x},q^{-x})} \bigg[{m \atop \kappa+1}\bigg]_{\mathcal{R}(p^{-x},q^{-x})}=[m-\kappa]_{\mathcal{R}(p^{-x},q^{-x})} \bigg[{m \atop \kappa}\bigg]_{\mathcal{R}(p^{-x},q^{-x})},
\end{equation*}
then, after computation, the recursion relation is found. 
\end{proof}
\begin{corollary}
	The recursion relation for the $q-$ P\'olya distributions is derived as:
	\begin{equation}
	P_{\kappa +1} ={q^{x(n+m-2\kappa-1)}} {[n-\kappa]_{q^{-x}}\over [u-n+\kappa+1]_{q^{-x}}}\,{[m-\kappa]_{q^{-x}}\over [\kappa+1]_{q^{-x}}}P_\kappa,
	\end{equation}
	with 
	\begin{equation}
	P_0= q^{-x\,m\,n}{[u]_{n, q^{-x})} \over [m+u]_{n, q^{-x}}}.
	\end{equation}
\end{corollary}
\begin{proof}
	It is straightforward.
\end{proof}
\begin{remark}
		 For $x=-1$,  we obtain the $\mathcal{R}(p,q)-$ hypergeometric distribution:
		\begin{eqnarray}\label{hd}
		P_r\big(T_n=\kappa\big)&=&\Phi(p,q)\bigg[\begin{array}{c} n  \\ \kappa\end{array} \bigg]_{\mathcal{R}(p,q)}\, {[m]_{\kappa,\mathcal{R}(p,q)}\,[u]_{n-\kappa,\mathcal{R}(p,q)}\over [m+u]_{n,\mathcal{R}(p,q)}},
		\end{eqnarray}
		where  $0<q<p\leq 1,$ $\kappa\in\{0,\cdots,n\},$ 
		$\Phi(p,q)={\epsilon^{\kappa(u-n+\kappa)}_1\, \epsilon^{(n-\kappa)(m-\kappa)}_2}$
		and
		$\displaystyle\sum_{\kappa=0}^{n}P_r( T_n=\kappa)=1.$
\end{remark}
The particular case of $q-$ deformation obtained in \cite{CA2} is characterized by
	the $j^{th}-$ order $q-$ deformed factorial moment:
	\begin{eqnarray}
	\mu_{q}\big([T_n]_{j, q^{-x}}\big) = {[n]_{j, q^{-x}}\,[m]_{j, q^{-x}}\over [m+u]_{j, q^{-x}}},\quad j\in\{1,\cdots,n\},
	\end{eqnarray}
and
the factorial moment:
	\begin{eqnarray}
	\mu_{q}[(T_n)_i] = i!\sum_{j=i}^{n}(-1)^{j-i}\bigg[\begin{array}{c} n  \\ j \end{array} \bigg]_{q}{ s_{q^{-x}}(j,i)\over (1-q^{-x})^{i-j}}{[m]_{j,q^{-x}}\over [m+u]_{j,q^{-x}}},
	\end{eqnarray} 
	where    
	 $ i\in\{1,2,\cdots,n\}$  and  $s_{q^{-x}}$ is the $q-$ deformed Stirling number of the first kind.
	 \begin{remark} The particular case of $q-$generalized Quesne deformation is characterized by the following properties:
	 	\begin{itemize}
	 		\item
	 		The   generalized $q-$ Quesne probability distribution:
	 		\begin{equation}
	 		P_r( T_n=\kappa)={q^{x(m-\kappa)(n-\kappa)}\over p^{x\,\kappa(u-n+\kappa)}}\bigg[{n \atop \kappa}\bigg]^Q_{p^{-x},q^{-x}}{[m]^Q_{\kappa,p^{-x},q^{-x}}[u]^Q_{n-\kappa,p^{-x},q^{-x}} \over [m+u]^Q_{n,p^{-x},q^{-x}}};
	 		\end{equation}
	 		\item
	 		The $j^{th}-$ order  generalized $q-$ Quesne factorial moment:
	 		\begin{eqnarray}
	 		\mu^Q_{p,q}\big([T_n]^Q_{j,p^{-x},q^{-x}}\big) =\Big({q\over p}\Big)^{-x\,j} {[n]^Q_{j,p^{-x},q^{-x}}\,[m]^Q_{j,p^{-x},q^{-x}}\over [m+u]^Q_{j,p^{-x},q^{-x}}},\, j\in\{1,2,\cdots, n\};
	 		\end{eqnarray}
	 		\item The  generalized $q-$ Quesne factorial moment:
	 		\begin{equation}
	 		\mu^Q_{p,q}[(T_n)_i] = i!\sum_{j=i}^{n}(-1)^{j-i}\bigg[\begin{array}{c} n  \\ j \end{array} \bigg]^Q_{p^{-x},q^{-x}}{s^Q_{p^{-x},q^{-x}}(j,i)\,\,P(j,i)},
	 		\end{equation}
	where 
	\begin{equation}
	P(j,i)={p^{j\choose 2}\,(p^{-x}-q^{x})^{j-i}\,[m]^Q_{j,p^{-x},q^{-x}}\over p^{\tau(j-i)}\,[m+u]^Q_{j,p^{-x},q^{-x}}},
	\end{equation}
	$\tau\in\mathbb{N},$ $i\in\{2,\cdots,n\}$ and $s^Q_{p^{-x},q^{-x}}$ is the  generalized $q-$ Quesne Stirling number of the first kind;
	\item
		The recursion relation for the  generalized $q-$ Quesne distributions:
		\begin{equation}
		P_{\kappa +1} ={  p^{x(u-n+2\kappa+1)}\over q^{x(n+m-2\kappa-1)}} {[n-\kappa]^Q_{p^{-x},q^{-x}}\over [u-n+\kappa+1]^Q_{p^{-x},q^{-x}}}\,{[m-\kappa]^Q_{p^{-x},q^{-x}}\over [\kappa+1]^Q_{p^{-x},q^{-x}}}P_\kappa,
		\end{equation}
		with
		\begin{equation}
		P_0= q^{x\,m\,n}{[u]^Q_{n, p^{-x},q^{-x}} \over [m+u]^Q_{n,p^{-x},q^{-x}}}.
		\end{equation}
	\end{itemize}
\end{remark}
\subsection{Inverse P\'olya distribution}
Let $Y_n$ be a number of black boxes drawn until the $n^{th}$ white box is drawn. Then, we have:
\begin{definition}
	The inverse $\mathcal{R}(p,q)-$ deformed P\'olya distribution with parameters $n,$ $m,$ $u,$ and $\kappa$ is given by:
	\begin{small}
	\begin{eqnarray}\label{ipa}
	{P_y:=}P_r( Y_n = y)= F(p,q)\bigg[{n+y-1\atop y}\bigg]_{\mathcal{R}(p^{-x},q^{-x})}{[m]_{n,\mathcal{R}(p^{-x},q^{-x})}\,[u]_{y,\mathcal{R}(p^{-x},q^{-x})}\over [m+u]_{n+y,\mathcal{R}(p^{-x},q^{-x})}},
	\end{eqnarray}
	\end{small}
	where  $0<q<p\leq 1,$ $y\in\{1,\cdots,n\},$  $F(p,q)=\epsilon^{n(u-x)}_1\,\epsilon^{-yx(m-n+1)}_2,$  $x$ is an integer,
	and
	\begin{equation*}\label{ipb}
	\sum_{t=0}^{\infty}F(p,q)\bigg[{n+y-1\atop y}\bigg]_{\mathcal{R}(p^{-x},q^{-x})}{[m]_{n,\mathcal{R}(p^{-x},q^{-x})}\,[u]_{y,\mathcal{R}(p^{-x},q^{-x})}\over [m+u]_{n+y,\mathcal{R}(p^{-x},q^{-x})}}=1.
	\end{equation*}
\end{definition}
	 Note that 
	the inverse $q-$ deformed P\'olya distribution  obtained in \cite{CA2}   can  be recovered by taking $\mathcal{R}(p,q)=q,$ $\epsilon_1=1,$ $\epsilon_2=q$ and $F(p,q)=q^{-yx(m-n+1)}$ as:
\begin{equation*}\label{qipa}
P_r( Y_n = y):= q^{-yx(m-n+1)}\bigg[{n+y-1\atop y}\bigg]_{q^{-x}}{[m]_{n, q^{-x}}\,[u]_{y,q^{-x}}\over [m+u]_{n+y,q^{-x}}}
\end{equation*}
where  $0<q< 1$ and  $y\in\{1,\cdots,n\}.$
\begin{theorem}
 The $j^{th}-$ order $\mathcal{R}(p,q)-$factorial moment is given by:
		\begin{equation}\label{ipc}
		{\mu}_{\mathcal{R}(p,q)}\big( [Y]_{j,\mathcal{R}(p^{-x},q^{-x})}\big) = {[n-j+1]_{j,\mathcal{R}(p^{-x},q^{-x})}[u]_{\mathcal{R}(j,p^{-x},q^{-x})}\over \epsilon^{jx(m-n+1)}_2[m+j]_{j,\mathcal{R}(p^{-x},q^{-x})}},
		\end{equation}
		where $j\in\mathbb{N}\backslash\{0\}$ and $m+j\neq 0.$
		Moreover, for $i\in\{1,2,\cdots,n\},$  the factorial moment yields:
		\begin{equation}\label{ipd}
		{ \mu}_{\mathcal{R}(p,q)}[(Y_n)_i] = i!\sum_{j=i}^{n}(-1)^{j-i}\bigg[\begin{array}{c} n+j-1  \\ j \end{array} \bigg]_{\mathcal{R}(p^{-x},q^{-x})}{IP(j,i)},
		\end{equation}
		where
		\begin{equation}
		IP(j,i)=s_{\mathcal{R}(p^{-x},q^{-x})}(j,i)\,{\epsilon^{j\choose 2}_1\over \epsilon^{\tau(j-i)}_1}{(\epsilon^{-x}_1-\epsilon^{-x}_2)^{j-i}[u]_{j,\mathcal{R}(p^{-x},q^{-x})}\over [m+j]_{j,\mathcal{R}(p^{-x},q^{-x})}\,\epsilon^{jx(m-n+1)}_2},
		\end{equation} 
		 $\tau\in\mathbb{N}$ and  $s_{\mathcal{R}(p^{-x},q^{-x})}$ is the $\mathcal{R}(p,q)-$ deformed Stirling number of the first kind.
The recursion relation for the inverse $\mathcal{R}(p,q)-$ P\'olya distibutions is provided by:
		\begin{equation}
		P_{y+1} = {\epsilon^{-x(m-n+1)}_2[n+y]_{\mathcal{R}(p^{-x},q^{-x})}[u-y]_{\mathcal{R}(p^{-x},q^{-x})}\over [y+1]_{\mathcal{R}(p^{-x},q^{-x})}[m+u-n-y]_{\mathcal{R}(p^{-x},q^{-x})}}\, P_y,
		\end{equation}
		with  initial condition
		\begin{equation}
		P_0 =\epsilon^{n(u-x)}_1 {[m]_{n,\mathcal{R}(p^{-x},q^{-x})}\over [m+u]_{{n,}\mathcal{R}(p^{-x},q^{-x})}}.	
		\end{equation} 
\end{theorem}
\begin{proof} Using (\ref{ipa}), the $j^{th}-$ order $\mathcal{R}(p,q)-$ deformed factorial moment  of $Y_n$ is expressed as follows:
	\begin{small}
\begin{eqnarray*}\label{ipe}
	{\mu}_{\mathcal{R}(p,q)}\big([Y_n]_{j,\mathcal{R}(p^{-x},q^{-x})}\big) &=&\sum_{y=j}^{\infty}[y]_{j,\mathcal{R}(p^{-x},q^{-x})}Pr( Y_n =y)\nonumber\\
	&=& {[n-j+1]_{j,\mathcal{R}(p^{-x},q^{-x})}[u]_{j,\mathcal{R}(p^{-x},q^{-x})}[m]_{n,\mathcal{R}(p^{-x},q^{-x})}\over \epsilon^{-n(u-x)}_1\epsilon^{yx(m-n+1)}_2}\nonumber\\
	&\times& \sum_{y=j}^{\infty}\bigg[{n+y-1\atop y-j}\bigg]_{\mathcal{R}(p^{-x},q^{-x})}{[u-j]_{y-j,\mathcal{R}(p^{-x},q^{-x})}\over [m+u]_{n+y,\mathcal{R}(p^{-x},q^{-x})}}.
\end{eqnarray*}
\end{small}
From the negative $\mathcal{R}(p,q)-$deformed Vandermonde's formula (\ref{ndv11}), we have  the relation
\begin{eqnarray*}\label{ipf}
{\mu}_{\mathcal{R}(p,q)}( [Y_n]_{j,\mathcal{R}(p^{-x},q^{-x})})&=&{[n-j+1]_{j,\mathcal{R}(p^{-x},q^{-x})}[u]_{j,\mathcal{R}(p^{-x},q^{-x})}\over \epsilon^{jx(m-n+1)}_2}\nonumber\\
&\times& {[m]_{n,\mathcal{R}(p^{-x},q^{-x})}\over [m+j]_{n+j,\mathcal{R}(p^{-x},q^{-x})}}.
\end{eqnarray*}
Since,
\begin{equation*}\label{ipg}
[m+j]_{n+j,\mathcal{R}(p^{-x},q^{-x})}=[m+j]_{j,\mathcal{R}(p,q)}[m]_{n,\mathcal{R}(p^{-x},q^{-x})},
\end{equation*}
 the relation (\ref{ipc}) holds. Furthermore, the relation (\ref{ipd}) is deduced using the $\mathcal{R}(p,q)-$ deformed factorial moment and (\ref{fm}). 
\end{proof}
\begin{corollary}
	The recursion relation for the inverse $q-$ P\'olya distributions is given by the formula:
	\begin{equation}
	P_{y+1} = {q^{-x(m-n+1)}[n+y]_{q^{-x}}[u-y]_{q^{-x}}\over [y+1]_{q^{-x}}[m+u-n-y]_{q^{-x}}}P_y,\quad\mbox{with}\quad P_0 = {[m]_{n,q^{-x}}\over [m+u]_{{n,}q^{-x}}}.
	\end{equation}
\end{corollary}

Finally, it is worthy to note that,  for the generalized $q-$Quesne quantum algebra, we have:
	\begin{itemize}
		\item 
 The inverse generalized $q-$ Quesne P\'olya  probability distribution: 
		\begin{equation*}\label{ipga}
		Pr( Y_n = y)= {p^{n(u-x)}\over q^{xy(m-n+1)}}\bigg[{n+y-1\atop y}\bigg]^Q_{p^{-x},q^{-x}}{[m]^Q_{n,p^{-x},q^{-x}}\,[u]^Q_{y,p^{-x},q^{-x}}\over [m+u]^Q_{n+y,p^{-x},q^{-x}}};
		\end{equation*}
		\item The $j^{th}-$ order generalized $q-$ Quesne factorial moment:
		\begin{equation*}
		{\mu}^Q_{p,q} [Y]^Q_{j,p^{-x},q^{-x}} = {[n-j+1]^Q_{j,p^{-x},q^{-x}}[u]^Q_{j,p^{-x},q^{-x}} \over q^{-j\,x(m-n+1)}[m+j]^Q_{j,p^{-x},q^{-x}}}\big({q\over p}\big)^{-x\,j}
		\end{equation*}
		with the associated factorial moment
		\begin{equation*}
		{ \mu}^Q_{p,q}[(Y_n)_i] = i!\sum_{j=i}^{n}(-1)^{j-i}\bigg[\begin{array}{c} n+j-1  \\ j \end{array} \bigg]^Q_{p^{-x},q^{-x}}{s^Q_{p^{-x},q^{-x}}(j,i)\,IP(j,i)},
		\end{equation*}
		where
		\begin{equation*}
		 IP(j,i)={p^{j\choose 2}\over p^{\tau(j-i)}}{(p^{-x}-q^{x})^{j-i}
			[u]^Q_{j,p^{-x},q^{-x}} \over
			q^{-j\,x(m-n+1)}
			[m+j]^Q_{j,p^{-x},q^{-x}}},
			\end{equation*}
		 $s^Q_{p^{-x},q^{-x}}$ is the generalized $q-$ Quesne Stirling number of the first kind and $\tau\in\mathbb{N};$
		\item The recursion relation for the inverse generalized $q-$ Quesne {P\'olya} distrubtions:
		\begin{small}
		\begin{eqnarray*}
		P_{y+1} = {q^{x(m-n+1)}[n+y]^Q_{p,q}[u-y]^Q_{p,q}\over [y+1]^Q_{p,q}[m+u-n-y]^Q_{p,q}}P_y,\quad\mbox{with}\quad P_0 = {p^{n(u-x)}\,[m]^Q_{n,p^{-x},q^{-x}}\over [m+u]^Q_{n,p^{-x},q^{-x}}}.
		\end{eqnarray*}
		\end{small} 
		\end{itemize}
\section{Concluding  remarks}
 $\mathcal{R}(p, q)$-deformed univariate  dicrete binomial, Euler, P\'olya and inverse P\'olya distributions, induced by  the $\mathcal{R}(p, q)$-deformed quantum algebras, have been formulated and discussed in a general framework. Distributions and main properties related to  the  generalized $q-$ Quesne algebra have been considered as illustration.
Results for known $q-$ deformed distributions have also been recovered as particular cases.
\par\bigskip\noindent
{\bf Acknowledgment.} This work is supported by TWAS Research Grant RGA No. 17 - 542 RG / MATHS / AF / AC \_G  -FR3240300147.
 MNH thanks Professor Kalyan Sinha for his constant solicitude, and Professor Aurel I. Stan  for his hospitality during the QP40 Conference. 
\bibliographystyle{amsplain}

\end{document}